\documentclass[11pt]{amsart}
\usepackage{amsmath, amssymb, amsthm, enumerate, graphicx,color,soul,verbatim,algorithmic,hyperref}

\usepackage{float}
\usepackage{setspace}
\usepackage{mathpazo}
\usepackage{indentfirst}  
\usepackage{fullpage}
\usepackage{thmtools}
\usepackage{thm-restate}
\usepackage{tikz}
\usetikzlibrary{arrows.meta}

\numberwithin{equation}{subsection}

\numberwithin{equation}{subsection}

\newtheorem{theorem}[subsection]{Theorem}
\newtheorem{lemma}[subsection]{Lemma}
\newtheorem{corollary}[subsection]{Corollary}

\newtheorem{prop}[subsection]{Proposition}

\theoremstyle{definition}
\newtheorem{definition}[subsection]{Definition}
\newtheorem{algorithm}[subsection]{Algorithm}

\newtheorem{example}[subsection]{Example}

\newtheorem{remark}[subsection]{Remark}

\newcommand{\TP}{\mathbb{TP}}

\newcommand{\conv}{\text{conv}}
\newcommand{\tconv}{\text{tconv}}
\newcommand{\diag}{\text{diag}}
\newcommand{\val}{\text{val}}

\usepackage{amsmath}
\usepackage{amsfonts}
\author{Leon Zhang}
\email{leonyz@math.berkeley.edu}
\title{Computing convex hulls in the affine building of $SL_d$}

\begin{document}
\maketitle

\begin{abstract}
We describe an algorithm for computing the convex hull of a finite collection of points in the affine building of $SL_d(K)$, for $K$ a field with discrete valuation. These convex hulls describe the relations among a finite collection of invertible matrices over $K$. As a consequence, we bound the dimension of the tropical projective space needed to realize the convex hull as a tropical polytope.
\end{abstract}

\section{Introduction}

Affine buildings are infinite simplicial complexes originally introduced by Tits to study the structure of simple Lie groups. They have since found use in a variety of other contexts, including arithmetic and algebraic geometry \cite{CHSW, KT}, optimization \cite{H}, and phylogenetics \cite{DT}.

We consider the affine building $\mathcal B_d$ associated to the group $SL_d(K)$ over a discrete valued field $K$. There is a natural notion of convex hull in $\mathcal B_d$, which provides a geometric data structure for the relations among $d\times d$ invertible matrices over $K$. Originally introduced by Faltings \cite{F}, this data structure underlies Mustafin varieties \cite{CHSW, HL} and can be used to study the fundamental group of certain 3-manifolds \cite{Su}. Joswig, Sturmfels, and Yu \cite[Algorithm 2]{JSY} give a procedure for computing such a convex hull in $\mathcal B_d$ as the standard triangulation of a tropical convex hull in some tropical projective space. However, their algorithm requires the enumeration of all lattice points in the convex hull under consideration, which can be difficult to implement and is expensive in practice. We devise an improved algorithm with time complexity bounded in the dimension of the building and the number of matrices spanning the convex hull, making it feasible for the first time to compute convex hulls in practice.

We briefly describe the structure of this manuscript. In Section \ref{preliminaries} we review the basics of convex lattice theory and tropical geometry that we rely on throughout. We review an algorithm for computing an apartment containing two vertices and develop its application to our problem in Section \ref{SA-basis-section}. We then describe our novel algorithm and prove its correctness in Section \ref{convex-hulls}. In Section \ref{min-conv-triangles-section} we discuss an improvement on the previous algorithm when computing the convex hull of three lattice classes. Our algorithms have been implemented over the rational function field as a Polymake extension \cite{GJ}. Algorithm \ref{enveloping-membrane} has also been implemented in Mathematica over the field of rational numbers with a $p$-adic valuation. This software and the code for the examples in this paper can be found at 
{our supplementary materials webpage}:

\begin{center}
{\url{https://math.berkeley.edu/~leonyz/code/convex-hulls}}
\end{center}

\subsection{Acknowledgements}

The author would like to thank the Max Planck Institute for Mathematics in the Sciences for its hospitality while working on this project. He was partially supported by a National Science Foundation Graduate Research Fellowship. The author is grateful to Jacinta Torres, Lara Bossinger, and Madeline Brandt for reading early drafts of this manuscript. He would also like to thank Michael Joswig and Lars Kastner for generous help on writing a Polymake extension, Petra Schwer for suggesting a geometric interpretation of Lemma \ref{covering-lemma}, and Bernd Sturmfels for much valuable discussion and feedback.

\section{Preliminaries}
\label{preliminaries}

\subsection{Convex hulls}

We begin by fixing notation and reviewing the setup of \cite{JSY}. Let $K$ be a field with discrete valuation $\val:K\to\mathbb Z\cup \{\infty\}$, let $R$ be its valuation ring with residue field $k$, and $\pi$ a uniformizer. Note that $K^d$ is an $R$-module in a natural way.

\begin{definition}
A \emph{lattice} $\Lambda$ is an $R$-submodule of $K^d$ generated by $d$ linearly independent vectors in $K^d$. We often represent a lattice by an invertible matrix whose columns generate the lattice.

Let $\Lambda_1, \Lambda_2\subseteq K^d$ be two lattices. We say that $\Lambda_1$ and $\Lambda_2$ are \emph{equivalent} if there exists $c\in K^*$ such that $\Lambda_1 = c\Lambda_2$, and we write $[\Lambda]$ for the equivalence class of the lattice $\Lambda$. We say that two equivalence classes of lattices are \emph{adjacent} if there exist representative lattices $\Lambda_1$ and $\Lambda_2$ respectively such that $\pi\Lambda_1\subseteq \Lambda_2\subseteq \Lambda_1$.

\end{definition}

\begin{definition}
Let $\mathcal B_d$ be the flag simplicial complex whose 0-simplices are equivalence classes of lattices in $K^d$ and whose 1-simplices correspond to adjacent equivalence classes. We call $\mathcal B_d$ the \emph{affine building} of $SL_d(K)$. 
\end{definition}

\begin{example}
Consider the building $\mathcal B_2$ for $K=\mathbb Q_3$ the $3$-adic numbers. In this case our valuation ring $R=\mathbb Z_3$ is the $3$-adic integers and our uniformizer $\pi$ is 3. The affine building $\mathcal B_2$ is the infinite tree with every vertex having degree 4 in Figure \ref{first-figure}.
\begin{figure}[h!]
\centering
\resizebox{0.375\textwidth}{!}{
\begin{tikzpicture}[baseline]
\filldraw (0,0) circle [radius=0.1];
\draw (.11, .11) -- (3, 3);
\draw (-0.11, 0.11) -- (-3, 3);
\draw (0.11, -0.11) -- (3, -3);
\draw (-0.11, -0.11) -- (-3, -3);
\filldraw (3,3) circle [radius=0.1];
\draw (3,3) -- (4, 2.5);
\draw (3,3) -- (2.5, 4);
\draw (3,3) -- (3.75, 3.75);
\filldraw (3,-3) circle [radius=0.1];
\draw (3,-3) -- (4, -2.5);
\draw (3,-3) -- (2.5, -4);
\draw (3,-3) -- (3.75, -3.75);
\filldraw (-3,3) circle [radius=0.1];
\draw (-3,3) -- (-4, 2.5);
\draw (-3,3) -- (-2.5, 4);
\draw (-3,3) -- (-3.75, 3.75);
\filldraw (-3,-3) circle [radius=0.1];
\draw (-3,-3) -- (-4, -2.5);
\draw (-3,-3) -- (-2.5, -4);
\draw (-3,-3) -- (-3.75, -3.75);
\filldraw (4,2.5) circle [radius=0.1];
\draw (4, 2.5) -- (4.5,2.5);
\draw (4, 2.5) -- (4.5,2.25);
\draw (4, 2.5) -- (4.5,2.75);
\filldraw (2.5,4) circle [radius=0.1];
\draw (2.5, 4) -- (2.5, 4.5);
\draw (2.5, 4) -- (2.25, 4.5);
\draw (2.5, 4) -- (2.75, 4.5);
\filldraw (3.75 ,3.75) circle [radius=0.1];
\draw (3.75, 3.75) -- (4.33, 4.33);
\draw (3.75, 3.75) -- (4.1, 4.45);
\draw (3.75, 3.75) -- (4.45, 4.1);
\filldraw (-4,-2.5) circle [radius=0.1];
\draw (-4, -2.5) -- (-4.5,-2.5);
\draw (-4, -2.5) -- (-4.5,-2.25);
\draw (-4, -2.5) -- (-4.5,-2.75);
\filldraw (-2.5,-4) circle [radius=0.1];
\draw (-2.5, -4) -- (-2.5, -4.5);
\draw (-2.5, -4) -- (-2.25, -4.5);
\draw (-2.5, -4) -- (-2.75, -4.5);
\filldraw (-3.75 ,-3.75) circle [radius=0.1];
\draw (-3.75, -3.75) -- (-4.33, -4.33);
\draw (-3.75, -3.75) -- (-4.1, -4.45);
\draw (-3.75, -3.75) -- (-4.45, -4.1);
\filldraw (-4,2.5) circle [radius=0.1];
\draw (-4, 2.5) -- (-4.5,2.5);
\draw (-4, 2.5) -- (-4.5,2.25);
\draw (-4, 2.5) -- (-4.5,2.75);
\filldraw (-2.5,4) circle [radius=0.1];
\draw (-2.5, 4) -- (-2.5, 4.5);
\draw (-2.5, 4) -- (-2.25, 4.5);
\draw (-2.5, 4) -- (-2.75, 4.5);
\filldraw (-3.75 ,3.75) circle [radius=0.1];
\draw (-3.75, 3.75) -- (-4.33, 4.33);
\draw (-3.75, 3.75) -- (-4.1, 4.45);
\draw (-3.75, 3.75) -- (-4.45, 4.1);
\filldraw (4,-2.5) circle [radius=0.1];
\draw (4, -2.5) -- (4.5,-2.5);
\draw (4, -2.5) -- (4.5,-2.25);
\draw (4, -2.5) -- (4.5,-2.75);
\filldraw (2.5,-4) circle [radius=0.1];
\draw (2.5, -4) -- (2.5, -4.5);
\draw (2.5, -4) -- (2.25, -4.5);
\draw (2.5, -4) -- (2.75, -4.5);
\filldraw (3.75 ,-3.75) circle [radius=0.1];
\draw (3.75, -3.75) -- (4.33, -4.33);
\draw (3.75, -3.75) -- (4.1, -4.45);
\draw (3.75, -3.75) -- (4.45, -4.1);
\end{tikzpicture}}
\hfill
\resizebox{0.6\textwidth}{!}{
\begin{tikzpicture}[baseline]
\filldraw (0, 0) circle [radius=0.1];
\filldraw (3,3) circle [radius=0.1];
\filldraw (-3,-3) circle [radius=0.1];
\filldraw (-3, 3) circle [radius=0.1];
\filldraw (3,-3) circle [radius=0.1];
\draw
	(0, 0) circle [fill=black, radius = 0.1] node[above = 3.75mm] {$\begin{pmatrix}1 & 0 \\ 0 & 1\end{pmatrix}$}
	(3,3) circle [radius = 0.1] node[right] {$\begin{pmatrix}1 & 0 \\ 0 & 3 \end{pmatrix}$}
--	(-3,-3) circle [radius = 0.1] node[left] {$\begin{pmatrix}1 & 0 \\ 1 & 3 \end{pmatrix}$}
	(3,-3) circle [radius = 0.1] node[right]  {$\begin{pmatrix}1 & 0 \\ 2 & 3 \end{pmatrix}$}
--	(-3,3) circle [radius = 0.1] node[left] {$\begin{pmatrix}3 & 0 \\ 0 & 1 \end{pmatrix}$};
\end{tikzpicture}}
\caption{Left: the building $\mathcal B_2$ for $K=\mathbb Q_3$. Right: the link of the identity in this building.}
\label{first-figure}
\end{figure}
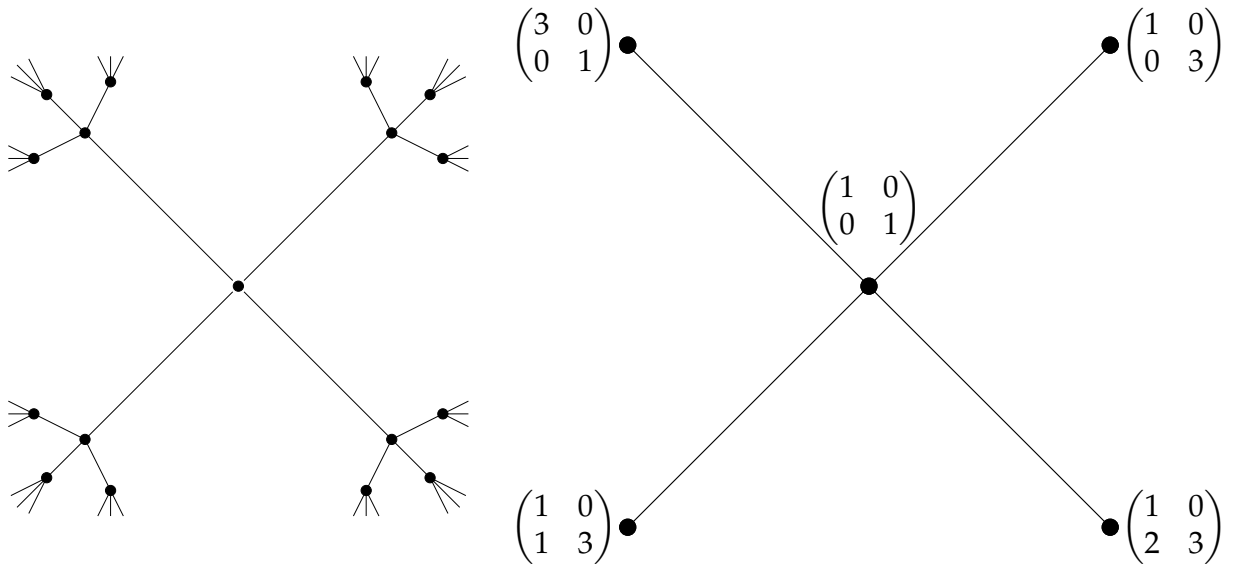

\end{example}

\begin{definition}
Let $M = (v_1, \dots, v_n)$ be a $d\times n$ matrix over $K$ with columns $v_1,\ldots, v_n$ spanning $K^d$ as a $K$-vector space, where $n>d$. The \emph{membrane} $[M]$ of $M$ is the collection of all lattice classes of the form $R\{\pi^{u_1}v_1,\ldots, \pi^{u_n}v_n\}$ for $u_i\in\mathbb Z$. If $M$ is square, so that $d = n$, we call the membrane $[M]$ an \emph{apartment}.
\end{definition}

\begin{lemma}[\cite{KT}, Lemma 4.13]
\label{membranes-apartments}
Let $M$ be a rank $d$ matrix over $K$ of size $d\times n$ with $n>d$. Then the membrane $[M]$ is the union of all apartments spanned by $d\times d$ invertible submatrices of $M$.
\end{lemma}

\begin{example}
\label{early-mem-example}
Let $K = \mathbb Q_2$, and consider the rank-2 building $\mathcal B_2$ over $K$. This is an infinite tree where every vertex has degree 3. Within $\mathcal B_2$, the matrix $\begin{pmatrix}1 & 0 & 1\\ 0 & 1 & 2\end{pmatrix}$ defines the membrane in Figure \ref{mem1}, a simplicial subtree of the building consisting of three infinite paths emanating from a vertex of degree 3.
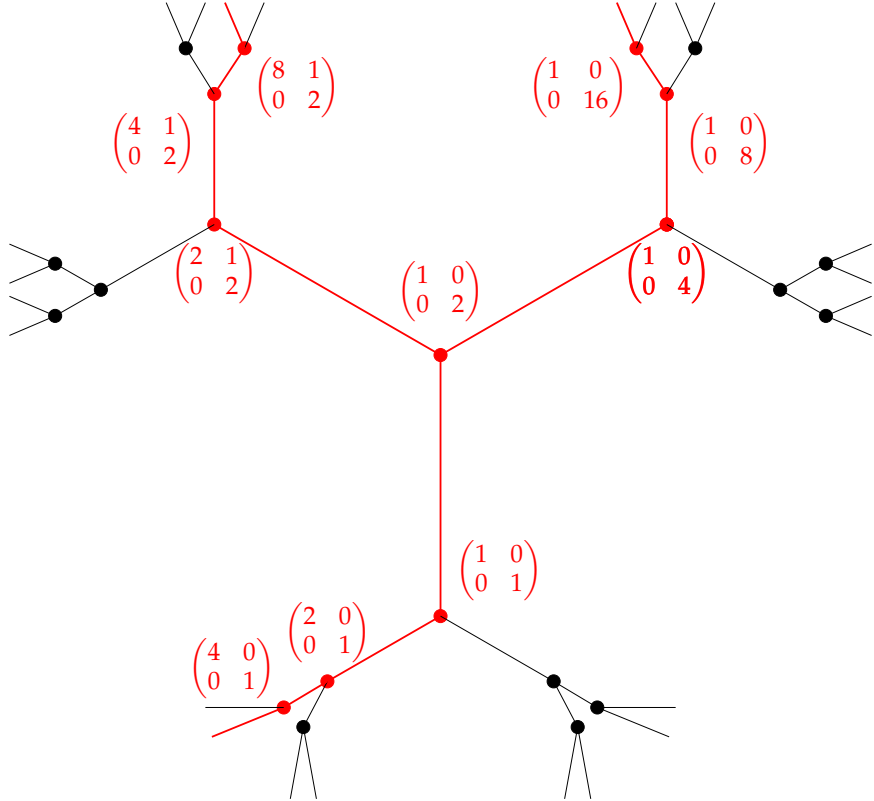
\begin{figure}[h!]
\centering
\resizebox{0.7\textwidth}{!}{
\begin{tikzpicture}
\filldraw[red] (0,0) circle [radius=0.1] node[above = 3.75mm] {$\begin{pmatrix}1 & 0 \\ 0 & 2\end{pmatrix}$};
\draw[red, thick] (0,0) -- (0, -4);
\draw[red, thick] (0,0) -- (3.464, 2);
\draw[red, thick] (0,0) -- (-3.464, 2);
\filldraw[red] (0,-4) circle [radius=0.1] node[above right = 1mm = 3.75mm] {$\begin{pmatrix}1 & 0 \\ 0 & 1\end{pmatrix}$};
\filldraw[red]  (-1.732, -5) circle [radius = 0.1] node[above=1.5mm] {$\begin{pmatrix}2 & 0 \\ 0 & 1\end{pmatrix}$};
\draw (-1.732, -5) -- (-2.1, -5.7);
\draw[red, thick] (-1.732, -5) -- (-2.4, -5.4);
\draw[red, thick] (0, -4) -- (-1.732, -5);
\draw (0, -4) -- (1.732, -5);
\filldraw  (-2.1, -5.7) circle [radius = 0.1];
\draw (-2.3, -6.8) -- (-2.1, -5.7);
\draw (-1.9, -6.8) -- (-2.1, -5.7);
\filldraw[red]  (-2.4, -5.4) circle [radius = 0.1] node[above left] {$\begin{pmatrix}4 & 0 \\ 0 & 1\end{pmatrix}$};
\draw[red, thick] (-2.4, -5.4) -- (-3.5, -5.85);
\draw (-2.4, -5.4) -- (-3.6, -5.4);
\filldraw  (1.732, -5) circle [radius = 0.1];
\filldraw  (2.1, -5.7) circle [radius = 0.1];
\draw (2.3, -6.8) -- (2.1, -5.7);
\draw (1.9, -6.8) -- (2.1, -5.7);
\filldraw  (2.4, -5.4) circle [radius = 0.1];
\draw (2.4, -5.4) -- (3.5, -5.85);
\draw (2.4, -5.4) -- (3.6, -5.4);
\draw (1.732, -5) -- (2.1, -5.7);
\draw (1.732, -5) -- (2.4, -5.4);
\filldraw[red] (3.464, 2) circle [radius=0.1] node[below = 1.5mm] {$\begin{pmatrix}1 & 0 \\ 0 & 4\end{pmatrix}$};
\filldraw[red] (3.464, 2) circle [radius=0.1] node[below = 1.5mm] {$\begin{pmatrix}1 & 0 \\ 0 & 4\end{pmatrix}$};
\filldraw[red] (3.464, 2) circle [radius=0.1] node[below = 1.5mm] {$\begin{pmatrix}1 & 0 \\ 0 & 4\end{pmatrix}$};
\filldraw[red] (-3.464, 2) circle [radius=0.1] node[below = 1.5mm] {$\begin{pmatrix}2 & 1 \\ 0 & 2\end{pmatrix}$};
\filldraw[red] (3.464, 4) circle [radius = 0.1] node[below right = 1.5mm]{$\begin{pmatrix}1 & 0 \\ 0 & 8\end{pmatrix}$};
\draw[red, thick] (3.464, 2) -- (3.464, 4);
\filldraw (5.2, 1) circle [radius = 0.1];
\draw (3.464, 2) -- (5.2, 1);
\filldraw[red] (3, 4.7) circle [radius = 0.1] node[below left ] {$\begin{pmatrix}1 & 0 \\ 0 & 16\end{pmatrix}$};
\draw[red, thick] (3.464, 4) -- (3,4.7);
\draw[red, thick](2.7, 5.4) -- (3, 4.7);
\draw(3.3, 5.4) -- (3, 4.7);
\draw(3.464, 4) -- (3.9, 4.7);
\draw(4.2, 5.4) -- (3.9, 4.7);
\draw(3.6, 5.4) -- (3.9, 4.7);
\filldraw (3.9, 4.7) circle [radius = 0.1];
\filldraw (5.9, 0.6) circle [radius = 0.1];
\filldraw (5.9, 1.4) circle [radius = 0.1];
\draw (5.2, 1) -- (5.9, 0.6);
\draw (5.2, 1) -- (5.9, 1.4);
\draw (6.6, 0.9) -- (5.9, 0.6);
\draw (6.6, 0.3) -- (5.9, 0.6);
\draw (6.6, 1.1) -- (5.9, 1.4);
\draw (6.6, 1.7) -- (5.9, 1.4);
\filldraw[red] (-3.464, 4) circle [radius = 0.1] node[below left = 1.5mm]{$\begin{pmatrix}4 & 1 \\ 0 & 2\end{pmatrix}$};
\draw[red, thick] (-3.464, 2) -- (-3.464, 4);
\filldraw (-5.2, 1) circle [radius = 0.1];
\draw (-3.464, 2) -- (-5.2, 1);
\filldraw[red] (-3, 4.7) circle [radius = 0.1] node[below right ] {$\begin{pmatrix}8 & 1 \\ 0 & 2\end{pmatrix}$};
\draw[red, thick] (-3.464, 4) -- (-3,4.7);
\draw(-2.7, 5.4) -- (-3, 4.7);
\draw[red, thick](-3.3, 5.4) -- (-3, 4.7);
\draw(-3.464, 4) -- (-3.9, 4.7);
\draw(-4.2, 5.4) -- (-3.9, 4.7);
\draw(-3.6, 5.4) -- (-3.9, 4.7);
\filldraw (-3.9, 4.7) circle [radius = 0.1];
\filldraw (-5.9, 0.6) circle [radius = 0.1];
\filldraw (-5.9, 1.4) circle [radius = 0.1];
\draw (-5.2, 1) -- (-5.9, 0.6);
\draw (-5.2, 1) -- (-5.9, 1.4);
\draw (-6.6, 0.9) -- (-5.9, 0.6);
\draw (-6.6, 0.3) -- (-5.9, 0.6);
\draw (-6.6, 1.1) -- (-5.9, 1.4);
\draw (-6.6, 1.7) -- (-5.9, 1.4);
\end{tikzpicture}}
\caption{A membrane, in red,  contained in the building $\mathcal B_2$ over $K = \mathbb Q_2$.}
\label{mem1}
\end{figure}
\end{example}

\begin{definition}
If $\Lambda_1$ and $\Lambda_2$ are lattices, then their intersection $\Lambda_1\cap \Lambda_2$ is also a lattice. We say that a collection of lattice classes is \emph{convex} if it is closed under taking intersections of a finite subset of representatives. 

Given a finite collection of lattices $\Lambda_1,\ldots, \Lambda_s$, we call their \emph{convex hull} $\conv(\Lambda_1, \dots, \Lambda_s)$ the smallest convex set containing their lattice classes. We can similarly define the convex hull of an infinite collection of lattices. In addition, given invertible matrices $M_1,\dots, M_s$, we write $\conv(M_1,\dots, M_s)$ for the convex hull $\conv(\Lambda_1,\dots,\Lambda_s)$ where each $\Lambda_i$ is the lattice spanned by the columns of $M_i$.
\end{definition}

\begin{remark}
Our notion of convexity corresponds to \emph{min-convexity} in the language of \cite{JSY}. There is another notion of convexity called \emph{max-convexity} which arises by considering sums of lattices instead of intersections. The duality functor $\Lambda\mapsto \Lambda^*=\text{Hom}_R(\Lambda, R)$ switches sums and intersections, so via this map the max-convex hull $\text{maxconv}(\Lambda_1,\dots, \Lambda_s)$ is isomorphic to $\conv(\Lambda_1^*,\dots, \Lambda_s^*)$. In particular, we may restrict our attention to convex hulls, and everything that follows can easily be translated to the language of max-convexity.
\end{remark}

\begin{example}
\label{early-convex-example}
Let $K = \mathbb Q_5$ and consider the matrices 
\[M_1 = \begin{pmatrix}1 & 0 & 0 \\ 0 & 1 & 0\\ 0 & 0 & 1\end{pmatrix}, M_2 = \begin{pmatrix} 1 & 0 & 0\\ 0 & \frac 1 5 & 0\\ 0 & 0 & \frac{1}{125}\end{pmatrix}, M_3 = \begin{pmatrix} 5 & 625 & 150 \\ 0 & 25 & 1\\ 0 & 0 & \frac 1 5\end{pmatrix}.\]
In Example \ref{conv-triangle-ex1} we will see that the convex hull $\conv(M_1,M_2,M_3)$ contains nine vertices, fifteen edges, and seven triangles, with the simplicial complex structure shown in Figure \ref{convex-triangle-diagram}.
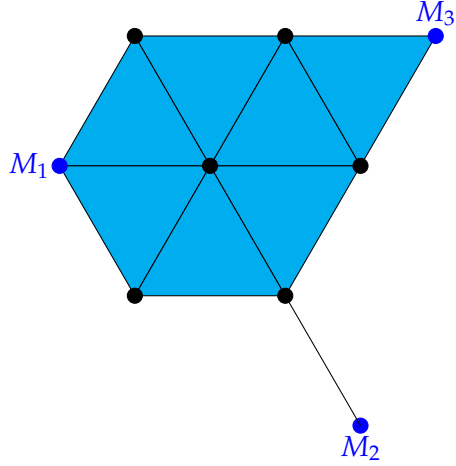
\begin{figure}[h!]
\begin{tikzpicture}
\filldraw[fill=cyan, draw=black] (0,0) -- (2,0) -- (1, 1.728) -- (0,0);
\filldraw[fill=cyan, draw=black] (0,0) -- (2,0) -- (1, -1.728) -- (0,0);
\filldraw[fill=cyan, draw=black] (3, 1.728) -- (2,0) -- (1, 1.728) -- (3, 1.728);
\filldraw[fill=cyan, draw=black] (0,0) -- (-1, 1.728) -- (1, 1.728) -- (0,0);
\filldraw[fill=cyan, draw=black] (0,0) -- (-1, 1.728) -- (-2, 0) -- (0,0);
\filldraw[fill=cyan, draw=black] (0,0) -- (-1, -1.728) -- (-2, 0) -- (0,0);
\filldraw[fill=cyan, draw=black] (0,0) -- (-1, -1.728) -- (1, -1.728) -- (0,0);
\filldraw (0,0) circle [radius = 0.1];
\filldraw[blue] (-2,0) circle [radius = 0.1] node[left] {$M_1$};
\filldraw (2,0) circle [radius = 0.1];
\filldraw (-1,-1.728) circle [radius = 0.1];
\filldraw (1,-1.728) circle [radius = 0.1];
\filldraw (-1,1.728) circle [radius = 0.1];
\filldraw (1,1.728) circle [radius = 0.1];
\filldraw[blue] (3,1.728) circle [radius = 0.1] node[above] {$M_3$};
\filldraw[blue] (2,-3.456) circle [radius = 0.1] node[below] {$M_2$};
\draw (1, -1.728) -- (2, -3.456);
\end{tikzpicture}
\caption{The convex hull of the matrices $M_1, M_2,$ and $M_3$ in $\mathcal B_3$.}
\label{convex-triangle-diagram}
\end{figure}
\end{example}

\begin{lemma}
\label{min-conv-decomposition}
Let $\Lambda_1,\dots, \Lambda_s$ be a finite collection of lattices. Then 
\[\conv(\Lambda_1,\dots, \Lambda_s) = \bigcup_{\Lambda'\in \conv(\Lambda_2,\dots,\Lambda_s)}\conv(\Lambda_1,\Lambda').\]
\end{lemma}
\begin{proof}
Pick any class $V$ in $\conv(\Lambda_1,\dots, \Lambda_s)$ with representative $\pi^{a_1}\Lambda_1\cap\pi^{a_2}\Lambda_2\cap\dots\cap \pi^{a_s}\Lambda_s$. Clearly $\Lambda'=\pi^{a_2}\Lambda_2\cap\dots\cap \pi^{a_s}\Lambda_s$ satisfies $[\Lambda']\in \conv(\Lambda_2,\dots,\Lambda_s)$, so $V\in \conv(\Lambda_1,\Lambda')$. Conversely, fix a lattice $\Lambda'=\pi^{a_2'}\Lambda_2\cap\dots\cap \pi^{a_s'}\Lambda_s$ representing a class in $\conv(\Lambda_2,\dots, \Lambda_s)$. Any class $V$ in $\conv(\Lambda_1,\Lambda')$ has a representative of the form $\pi^{a_1}\Lambda_1 \cap \pi^{b}\Lambda'= \pi^{a_1}\Lambda_1\cap\pi^{a_2}\Lambda_2\cap\dots\cap \pi^{a_s}\Lambda_s$, where $a_i = b+a_i'$ for $i=2,\dots, s$. In particular, $V$ is certainly in $\conv(\Lambda_1,\dots,\Lambda_s)$.
\end{proof}

The following result was originally stated in Faltings's paper on matrix singularities \cite{F}. For completeness we provide an easy proof.

\begin{prop}
\label{finiteness}
Let $\Lambda_1,\ldots, \Lambda_s$ be a finite collection of lattices representing equivalence classes in $\mathcal B_d$. Then $\conv(\Lambda_1,\ldots, \Lambda_s)$ is finite.
\end{prop}
\begin{proof}
Any class in $\conv(\Lambda_1,\Lambda_2)$ has a representative of the form $\Lambda_1\cap\pi^a\Lambda_2$. For $a\gg 0$ we know that $\Lambda_1\supseteq \pi^a\Lambda_2$, and for $a\ll 0$ we know $\Lambda_1\subseteq \pi^a\Lambda_2$. Hence the convex hull of two lattices is finite. The result then follows by induction and Lemma \ref{min-conv-decomposition}.
\end{proof}

It is therefore natural to ask how to compute a convex hull. In fact, the building $\mathcal B_d$ and membranes have an innate tropical structure which can be exploited for this purpose.

\subsection{Tropical basics}
We next review some basics of tropical convexity. For a more detailed exposition of this material, see \cite[Chapter 4]{MS} or \cite[Chapter 5]{J}.

We work over the tropical semiring $\mathbb{T} = (\mathbb R\cup \{\infty\}, \oplus, \odot)$ with the min-plus convention. In this semiring, the basic arithmetic operations of addition and multiplication are redefined:
\[a\oplus b:=\min(a,b), \, a\odot b := a + b \text{\ \ where }a, b\in\mathbb T.\]

The \emph{tropical projective space} $\TP^{n-1}$ is the space $(\mathbb T^n - (\infty, \infty, \ldots, \infty))/{\bf 1 \mathbb R}$, where $\textbf 1=(1,1, \ldots,1)$ is the all-ones vector. When illustrating this space, as in Figure \ref{standard-triangulation-TP2}, we always choose the affine chart in which the first coordinate is 0. There is a tropical distance metric $d_{tr}$ in $\mathbb {TP}^{n-1}$, given by
\[d_{tr}(v,w) := \max\{|v_i-v_j+w_j-w_i|: 1\leq i < j \leq n\}.\]
The lattice of integral points $\mathbb Z^n\subseteq \TP^{n-1}$ forms the skeleton of a flag simplicial complex, with a 1-simplex between two lattice points if they are of tropical distance 1 apart. This is called the \emph{standard triangulation} of $\TP^{n-1}$.

\begin{figure}[h!]
\begin{tikzpicture}
\foreach \x in {-1,...,1}{
	\foreach \y in {-1, ..., 1} {
		\filldraw (2*\x,2*\y) circle [radius = 0.1];
	}
}
\foreach \x in {-1,...,1}{
	\draw (2*\x, -3) -- (2*\x, 3);
}
\foreach \y in {-1,...,1}{
	\draw (-3, 2*\y) -- (3, 2*\y);
}
\draw (1, -3) -- (3, -1);
\draw (-1, -3) -- (3, 1);
\draw (-3, -1) -- (1, 3);
\draw (-3, 1) -- (-1, 3);
\draw (-3,-3) -- (3,3);
\filldraw[cyan] (0,0) circle [radius = 0.1] node[below right] {$(0,0,0)$};
\end{tikzpicture}
\caption{The standard triangulation of $\mathbb{TP}^2$, with the origin colored cyan.}
\label{standard-triangulation-TP2}
\end{figure}
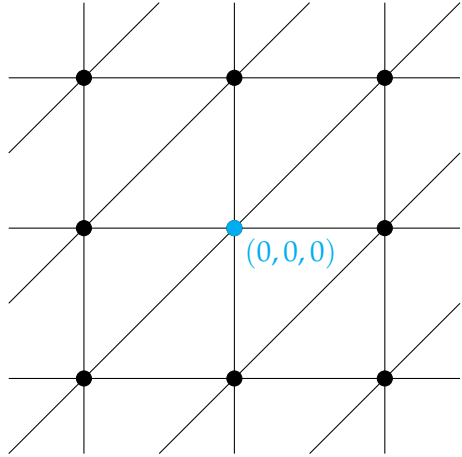

Given a collection $\mathcal P$ of points in $\mathbb{TP}^{n-1}$, we define their \emph{tropical convex hull} or \emph{tropical polytope} as the tropical semimodule spanned by these points, i.e.:
\[\tconv(\mathcal P) = \{\lambda_1\odot p^{(1)}\oplus\dots\oplus \lambda_s\odot p^{(s)}: \lambda_i\in\mathbb T, p^{(i)}\in \mathcal P\}.\]
If $\mathcal P\subseteq \mathbb Z^n$, we call their tropical convex hull a \emph{tropical lattice polytope}.

A map $p$ from the set of all $d$-sized subsets of $[n]$ to $\mathbb T$ satisfying the following exchange relation is called a \emph{valuated matroid} \cite{DT} or \emph{tropical Pl\"ucker vector}: for any $(d-1)$-subset $\sigma$ and any $(d+1)$-subset $\tau$ of $[n]$, the minimum
\[\min\{p(\sigma\cup\{\tau_i\}) + p(\tau - \{\tau_i\}): i\in [d+1]\}\]
is attained at least twice. (By convention we say that $p(\sigma)=\infty$ if $\sigma$ has size less than $d$.)

A tropical Pl\"ucker vector $p$ gives rise to a \emph{tropical linear space} $L$, consisting of all points $x\in \TP^{n-1}$ such that, for any $(d+1)$-subset $\tau$ of $[e]$, the minimum of the numbers $p(\tau-\{\tau_i\})+ x_{\tau_i}$, for $i = 1,\ldots, d, d+1$, is attained at least twice. Given a tropical linear space $L$, there is a projection map $pr_{L}$ taking a point $x\in \mathbb{TP}^{n-1}$ to a nearest point $pr_{L}(x)\in L$, which can be evaluated via the \emph{red rule} or \emph{blue rule} \cite[Theorem 15]{JSY}. We state here the blue rule, which gives a formula for the $i$th coordinate of $pr_{L}(x)$ in terms of an optimization over all $(d-1)$-sized subsets $\sigma$ of $[n]$:
\[pr_{L}(x)_i = \min_\sigma \max_{j\not\in \sigma}(p(\sigma\cup\{i\})-p(\sigma\cup\{j\})+x_j).\]
The standard triangulation of $\mathbb{TP}^{n-1}$ descends to a standard triangulation of any tropical convex hull of lattice points or of any tropical linear space $L$ with Pl\"ucker vector image in $\mathbb Z\cup\{\infty\}$. 

One important class of tropical linear spaces arises as follows. Let $M = (v_1,\ldots, v_n)$ be a $d\times n$ matrix of rank $d$ over $K$. Then $M$ defines a tropical Pl\"ucker vector $p$ (and hence tropical linear space $L$) as follows: if $\omega$ is a collection of $d$ integers in $[n]$, then $M_\omega$ denotes the corresponding $d\times d$ submatrix, and $p(\omega):=\val(\det(M_\omega))$. A tropical linear space obtained in this way is called a \emph{tropicalized linear space}.

\begin{example}
Let $M$ be the matrix over $\mathbb Q_2$ from Example \ref{early-mem-example}: $M = \begin{pmatrix}1 & 0 & 1\\ 0 & 1 & 2\end{pmatrix}$. We can compute the Pl\"ucker vector $p$ coming from $M$: 
\[p(\{1,2\}) = 0, p(\{1,3\})=1, p(\{2,3\})=0.\]
The corresponding tropical linear space $L$ consists of three rays emanating from the apex point $(0, -1, 0)$ in the $(0,0,1), (0,1,0),$ and $(1,0,0)$ directions of tropical projective space.
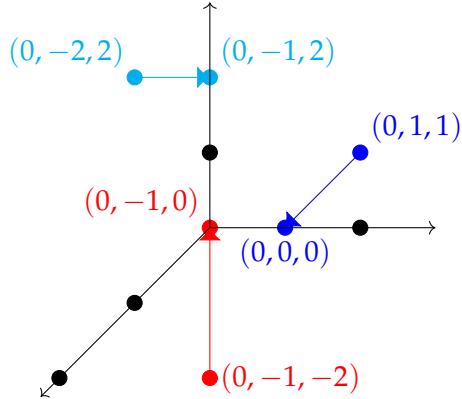
\begin{figure}[h!]
\centering
\begin{tikzpicture}
\filldraw[red] (0,1) circle [radius = 0.1] node[above left] {$(0, -1, 0)$};
\filldraw[red] (0,-1) circle [radius = 0.1] node[right] {$(0, -1, -2)$};
\filldraw (0,2) circle [radius = 0.1];
\filldraw[cyan] (0,3) circle [radius = 0.1] node[above right] {$(0, -1, 2)$};
\filldraw[blue] (1,1) circle [radius = 0.1] node[below] {$(0, 0,0)$};
\filldraw (2,1) circle [radius = 0.1];
\filldraw (-1,0) circle [radius = 0.1];
\filldraw (-2, -1) circle [radius = 0.1];
\filldraw[cyan] (-1, 3) circle [radius = 0.1] node[above left] {$(0, -2, 2)$};
\filldraw[blue] (2,2) circle [radius = 0.1] node[above right] {$(0, 1,1)$};
\draw [->] (0,1) -- (0,4);
\draw [->] (0,1) -- (-2.25, -1.25);
\draw [->] (0,1) -- (3,1);
\draw[cyan] [-{Latex[width=3mm]}] (-1, 3) -- (0,3);
\draw[blue] [-{Latex[width=3mm]}] (2,2) -- (1,1);
\draw[red] [-{Latex[width=3mm]}] (0,-1) -- (0,1);
\end{tikzpicture}
\caption{
The standard triangulation of the tropical linear space $L$ coming from the matrix $M$, along with projections of three lattice points not in $L$.
}
\label{tls1}
\end{figure}
\end{example}

We can now describe the tropical structure underlying the building $\mathcal B_d$. In effect, membranes are just standard triangulations of tropicalized linear spaces.

\begin{theorem}[\cite{KT}, Theorem 4.15]
\label{membrane-isomorphism-theorem}
Let $M = (v_1,\ldots, v_n)$ be a $d\times n$ matrix of rank $d$ over $K$ and let $L$ be its associated tropical linear space. Then there is an isomorphism $\Psi_M$ between the membrane $[M]$ and the standard triangulation of $L$,
\[\Psi_M(R\{\pi^{-u_1}v_1,\ldots, \pi^{-u_n}v_n\}) := pr_{L}((u_1,\ldots, u_n)),\]
sending a lattice $R\{\pi^{-u_1}v_1,\ldots, \pi^{-u_n}v_n\}$ to the projection onto $L$ of the point $(u_1, \dots, u_n)\in\mathbb{TP}^{n-1}$.
\end{theorem}

As a first illustration of this theorem, note that Figures \ref{mem1} and \ref{tls1} are isomorphic as simplicial complexes. They are both trees comprising three infinite branches stemming from a single node.

\begin{example}
If our matrix $M$ is square, so that its membrane $[M]$ is actually an apartment in the building, then $\Psi_M$ describes a simplicial complex isomorphism between the apartment and the \emph{tropical projective torus} $\mathbb R^{n}/\mathbb R{\bf 1}$.
\end{example}

\begin{example}
Keep the notation of Theorem \ref{membrane-isomorphism-theorem}. Rinc\'on \cite{R} describes a local structure of any tropical linear space $L$ with Pl\"ucker vector $p$, in which a basis $\sigma$ of the underlying matroid yields a \emph{local tropical linear space} defined by
\[L_\sigma = \{u\in L: p(\sigma) - \sum_{i\in \sigma} u_i \leq p(\tau) - \sum_{j\in \tau} u_j\, \forall \text{ bases } \tau\}.\] 
These local tropical linear spaces $L_\sigma$ are isomorphic to Euclidean space $\mathbb R^{n-1}$, are contained in $L$, and together form a non-disjoint cover of $L$.

The covering of a membrane by its apartments derives from this local structure of tropical linear spaces. In particular, let $\sigma$ describe a basis of the matroid of $M$, so that the $d\times d$ matrix $M_\sigma$ with columns indexed by $\sigma$ is invertible. Then the linearity of the determinant over column sums implies that the apartment $[M_\sigma]$ is mapped by $\Psi_M$ to the local tropical linear space $L_\sigma$.
\end{example}

Given any membrane $[M]$ represented by a $d\times n$ matrix $M = (f_1, \ldots, f_n)$, there is a \emph{retraction} $r_M$ of the entire building $\mathcal B_d$ onto $[M]$, which restricts to the identity on $[M]$ itself:
\[r_M: \Lambda\mapsto (\Lambda\cap K \{f_1\}) + \cdots + (\Lambda \cap K\{f_n\}).\]
We may use this map to describe a tropical structure for convex hulls.

\begin{theorem}[\cite{JSY}, Proposition 22]
\label{minconv-isomorphism-theorem}
Let $M$ be a $d\times n$ matrix of rank $d$ over $K$, $[M]$ its corresponding membrane, and $L$ its corresponding tropical linear space. Also let $\Lambda_1,\ldots, \Lambda_s$ be lattices corresponding to points in $\mathcal B_d$. The following two simplicial complexes coincide:
\[r_M(\conv(\Lambda_1,\ldots, \Lambda_s))\subseteq [M],\]
\[\tconv(\Psi_M(r_M(\Lambda_1)),\ldots, \Psi_M(r_M(\Lambda_s)))\subseteq L.\]
\end{theorem}

In particular, if $[M]$ contains the convex hull of $\Lambda_1,\ldots, \Lambda_s$, then the retraction map acts as the identity, and the convex hull $\conv(\Lambda_1,\ldots,\Lambda_s)$ is isomorphic to the standard triangulation of a tropical polytope. This suggests an approach for computing convex hulls in $\mathcal B_d$ as follows:
\vspace{.1in}

\begin{algorithm}[Convex hull computation]
\,
\label{convex-hull-computation}
\begin{algorithmic}[1]
\item[]
\REQUIRE $M_1, \ldots, M_s$ $d\times d$ invertible matrices over $K$ whose columns are bases for lattices $\Lambda_1,\ldots, \Lambda_s$
\ENSURE $\conv(\Lambda_1,\ldots, \Lambda_s)\subseteq \mathcal B_d$
\STATE $[M]\leftarrow$ a membrane containing $\conv(\Lambda_1,\ldots, \Lambda_s)$
\FORALL{$i\in \{1,\ldots, s\}$}
	\STATE $P_i \leftarrow \Psi_M(r_M(\Lambda_i))$
\ENDFOR
\STATE $X\leftarrow \tconv(P_1,\ldots, P_s)$
\RETURN{$X$}
\end{algorithmic}
\end{algorithm}

Note that given a lattice $\Lambda_i$ represented by a matrix $M_i$, we can compute the image $\Psi_M(r_M(\Lambda_i))$ in $\mathbb{TP}^{n-1}$ simply by taking the tropical row sum of the matrix $\val(M_i^{-1}M)$, by \cite[Lemma 21]{JSY}.

\begin{example}
\label{conv-triangle-ex1}
Retain the setup of Example \ref{early-convex-example}. We can apply Algorithm \ref{enveloping-membrane}, which we will soon discuss, to get that the membrane $[M]$ with 
\[M = \begin{pmatrix}1 & 0 & 0 & 0\\ 0 & 1 & 0 & 5\\ 0 &0 & 1 & 1\end{pmatrix}
\]
contains the convex hull of $M_1, M_2,$ and $M_3$. Running through Algorithm \ref{convex-hull-computation} with this membrane yields the following tropical matrix, 
\[
\begin{pmatrix}0 & 0 & 0 & 0\\ 0 & 1 & 2 & 3\\ -1 & -2 & 1 & -1\end{pmatrix},\]
whose rows or columns span the tropical convex hull in Figure \ref{convex-triangle-diagram}. Note that tropical polytopes are self-dual, i.e. the columns and rows of any tropical matrix span isomorphic tropical polytopes \cite[Theorem 1]{DS}.
\begin{figure}[h]
\centering
\includegraphics[height=3.2in]{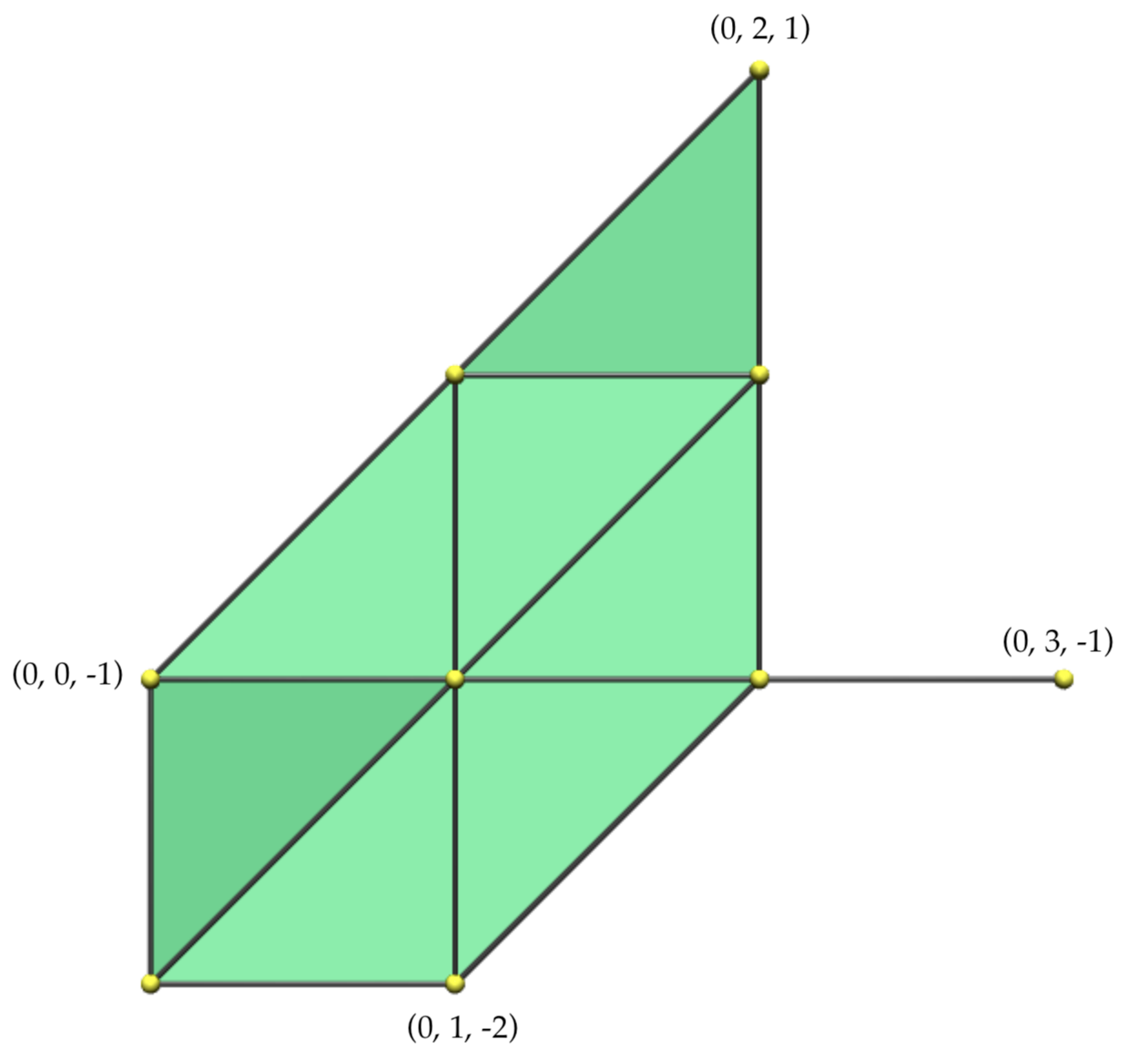}
\caption{The tropical polytope isomorphic to the convex hull of the three matrices in Example \ref{early-convex-example}, with coordinates of spanning vertices labeled.}
\label{tconv0pic}
\end{figure}

From this tropical polytope we can construct representatives for any of the lattice classes in Figure \ref{convex-triangle-diagram}. For example, consider the central lattice point $(0, 1, -1)$ with six neighbors. By Equation (14) in \cite{DS}, this lattice point in the column-span of our tropical matrix corresponds to $(-1, -1, 0, 0)$ in the row-span. In turn, this point corresponds to the class of the lattice
\[\mathbb Z_5\left\{5^{1}\cdot \begin{pmatrix}1 \\ 0 \\ 0\end{pmatrix}, 5^{1}\cdot \begin{pmatrix} 0 \\ 1 \\ 0\end{pmatrix}, 5^0 \cdot \begin{pmatrix} 0 \\ 0 \\ 1\end{pmatrix}, 5^0\cdot \begin{pmatrix}0 \\ 5 \\ 1\end{pmatrix}\right\}.\]
\end{example}

Of course, Algorithm \ref{convex-hull-computation} requires an \emph{enveloping membrane} of $\Lambda_1,\ldots, \Lambda_s$: a membrane containing the convex hull of $\Lambda_1,\ldots,\Lambda_s$. Without such a membrane the tropical polytope produced by Algorithm \ref{convex-hull-computation} need not be isomorphic to our original convex hull. Note that membranes need not be convex, so it is not sufficient to simply find a membrane containing the spanning lattice classes. This fact will be demonstrated in Example \ref{conv-quadrilateral-ex}.

No bounded-time algorithm for computing an arbitrary enveloping membrane has thus far been described. A procedure for computing such an enveloping membrane is described in \cite{JSY}, but it is often impractical: it relies on the computation of each individual element of the convex hull, while expanding a starting membrane to contain each element whenever necessary. In Section \ref{convex-hulls} we will describe an improved algorithm with bounded complexity in $d$ and $s$, allowing us to algorithmically realize any convex hull as a tropical polytope.

\begin{remark}
Our notion of convexity was originally introduced by Faltings \cite{F}, who noted that configurations $\Gamma$ of vertices in $\mathcal B_d$ correspond to certain schemes $M(\Gamma)$ called \emph{Mustafin varieties} or \emph{Deligne schemes} over the spectrum of a DVR. In the arithmetic setting, these Mustafin varieties function as local models of Shimura varieties \cite{PRS}. The special fiber of $M(\Gamma)$ generally has many singularities, but replacing $\Gamma$ with its convex hull $\Gamma'$ yields a regular Mustafin variety $M(\Gamma')$ with a dominant morphism $M(\Gamma')\to M(\Gamma)$, such that the irreducible components of the special fiber of $M(\Gamma')$ intersect transversally \cite[Lemma 2.4 and Theorem 2.10]{CHSW}. In this way, our fully-specified version of Algorithm \ref{convex-hull-computation}, derived in Section \ref{convex-hulls}, allows for the explicit resolution of singularities of Mustafin varieties.
\end{remark}

\section{Simultaneously-adaptable bases}

\label{SA-basis-section}

In this section we review a classical result on lattices over valued fields, following \cite[Section 6.9]{AB}, and describe its relevance to our setting of convex hulls in affine buildings. Note that in what follows we say \emph{monomial matrix} to refer to any matrix $A$ supported on a permutation matrix: i.e., there exists a permutation $\sigma$ such that $A_{ij}\neq 0$ if and only if $j=\sigma(i)$.

\vspace{.1in}
\begin{algorithm}[Simultaneously adaptable basis for two lattices]
\,
\label{CommonBasis}
\begin{algorithmic}[1]
\item[]
\REQUIRE $M_1, M_2$ $d\times d$ invertible matrices over $K$ whose columns are bases for lattices $\Lambda_1$ and $\Lambda_2$ in $\mathcal B_d$
\ENSURE Invertible matrix $A$, monomial matrix $\Delta$ such that the columns of $A$ and $A\Delta$ are bases for $\Lambda_1$ and $\Lambda_2$, respectively
\STATE $B_{1}\leftarrow M_1$
\STATE $C_{1}\leftarrow M_2$
\FORALL{$i \in \{1,\dots, n-1\}$}
\STATE $N_i \gets B_i^{-1}C_i$
\STATE $n_i\leftarrow $ entry of minimal valuation in $N_i$ not equal to $n_1, \ldots, n_{i-1}$
\STATE $L_i\gets$ the matrix such that $L_iN_i$ is obtained from $N_i$ by eliminating all other elements in the column of $n_i$
\STATE $R_i\gets$ the matrix such that $L_i N_i R_i$ is obtained from $L_i N_i$ by eliminating all other elements from the row of $n_i$
\STATE $B_{i+1}\gets B_i L_i^{-1}$
\STATE $C_{i+1}\gets C_iR_i$
\ENDFOR
\STATE $A\gets B_d$
\STATE $\Delta\gets B_d^{-1}C_d$
\RETURN{$A, \Delta$}
\end{algorithmic}
\end{algorithm}
\vspace{-0.075in}
\begin{lemma}
\label{CommonBasisProof}
Let $M_1$ and $M_2$ be $d\times d$ invertible matrices over $K$ for lattices $\Lambda_1$ and $\Lambda_2$ in $\mathcal B_d$. Then Algorithm \ref{CommonBasis} correctly returns a basis $A$ for $\Lambda_1$ and a monomial matrix $\Delta$ such that $A\Delta$ is a basis for $\Lambda_2$.
\end{lemma}
\begin{proof}
\vspace{-0.1in}
Because $n_i$ is chosen to be of minimal valuation in $N_i$, each $L_i$ and $R_i$ will be matrices in $SL_d(R)$. It follows that the new matrices $B_{i+1}$ and $C_{i+1}$ will be bases for $\Lambda_1$ and $\Lambda_2$ if $B_i$ and $C_i$ are, with base change matrix $L_iN_iR_i$. In particular, after $d-1$ steps of the for-loop, the matrix $\Delta = L_{d-1}N_{d-1}R_{d-1}$ will have $d-1$ distinct entries which are uniquely nonzero in their respective rows and columns. Hence $\Delta$ is a monomial matrix, as desired.
\end{proof}

\begin{definition}
We call the output $A$ in Algorithm \ref{CommonBasis} a \emph{simultaneously adaptable basis} (\emph{SA-basis}) for $\Lambda_1$ and $\Lambda_2$.
\end{definition}

Algorithm \ref{CommonBasis} is a demonstration of the building-theoretic fact that any two lattices lie in a common apartment. In general, there should be many distinct apartments containing any two given points. Indeed, the SA-basis $A$ obtained above depends not only on the lattices $\Lambda_1$ and $\Lambda_2$ but on our original choice of bases, which lattice we designate as $\Lambda_1$, and how we break "ties" between elements of minimal valuation when choosing pivots. In what follows, we break ties between potential pivots by picking the option in the leftmost column, then topmost row.

Because apartments are min-convex, we have the following fact:
\begin{corollary}
\label{apartment-containing-convex-hull}
Pick two lattice classes represented by $\Lambda_1$ and $\Lambda_2$ in $\mathcal B_d$, and let $A$ be a SA-basis for the two lattices. Then the apartment $[A]$ contains the convex hull $\conv(\Lambda_1,\Lambda_2)$.
\end{corollary}

We can therefore view Algorithm \ref{CommonBasis} as a procedure for computing an apartment containing the convex hull of two points.

\begin{lemma}
\label{persistent-scaling}
Let $M_1$ and $M_2$ be two invertible $d\times d$ matrices representing lattices $\Lambda_1$ and $\Lambda_2$, and let $\Gamma$ be any diagonal matrix with $M_2' = M_2\Gamma$. Let $N_i$ and $N_i'$ be the base-change matrices at the $i$th step of Algorithm \ref{CommonBasis} executed with the pairs $(M_1, M_2)$ and $(M_1, M_2')$ as input respectively, $n_i$ and $n_i'$ the chosen pivots of least valuation at step $i$, and so on. If $k$ is a positive integer such that the positions of the pivots $n_i$ and $n_i'$ agree for all $i$ up to $k-1$, then $L_{k-1} = L_{k-1}'$ and $N_k' = N_k\Gamma$.
\end{lemma}
\begin{proof}
We prove the result by induction, noting that the base case $k=1$ follows trivially. Suppose that the first $k-1$ pivots are the same for the two algorithm executions. Because the first $k-2$ pivots are the same, by the inductive hypothesis we have that $N_{k-1}'=N_{k-1}\Gamma$. In particular, the ratio of any two entries in the same column is the same for $N_{k-1}$ and $N'_{k-1}$. Now since the $k-1$st pivot position is also the same, the row operations to obtain $N_{k}$ and $N_{k}'$ from $N_{k-1}$ and $N_{k-1}'=N_{k-1}\Gamma$ agree as well, so that $L_{k-1} = L_{k-1}'$. Next the column operations necessary to clear the rows of two pivots may differ, but in both executions we eliminate using a column which has no other nonzero entries. It follows that $N_k' = N_k\Gamma$, as desired.
\end{proof}

\begin{corollary}
\label{pivot-positions}
Keep the setup of Lemma \ref{persistent-scaling} above, and let $A$ be the basis of $\Lambda_1$ produced by Algorithm \ref{CommonBasis}. If {all} pivot positions of Algorithm \ref{CommonBasis} are the same for the two inputs $(M_1, M_2)$ and $(M_1, M_2')$, then the lattice class $[\Lambda_2']$ for $M_2'$ is contained in the apartment $[A]$.
\label{apartment-criterion}
\end{corollary}
\begin{proof}
Because all pivots are the same, Lemma \ref{persistent-scaling} implies that $L_i = L_i'$ for all $i$. This means the final basis for $\Lambda_1$ produced by both executions of the algorithm is $A=L_d L_{d-1}\dots L_1 M_1$. In particular, we have that $[A]$ contains both $[\Lambda_2]$ and $[\Lambda_2']$.
\end{proof}

\begin{lemma}
\label{covering-lemma}
Let $A$ be an invertible matrix defining an apartment $[A]$ and $M$ a basis for a lattice $\Lambda$ whose class is not in $[A]$. Then $\conv([\Lambda], [A])$ can be covered with $d!$ different apartments.
\end{lemma}
\begin{proof}
Any element of $\conv([\Lambda],[A])$ will be contained in some $\conv([\Lambda], [\Lambda'])$ where $\Lambda'$ is a lattice having basis $A\Gamma$ for some diagonal matrix $\Gamma$. Fix such a $\Gamma$. We can use Algorithm \ref{CommonBasis} to compute an apartment $[B]$ containing $\conv([\Lambda], [A\Gamma])$. By Corollary \ref{pivot-positions}, this apartment also contains $\conv([\Lambda], [A\Gamma'])$ for any other diagonal matrix $\Gamma'$ leading to the same sequence of pivot positions. The key point is that we need only consider the sequence of columns that the pivots appear in. Since each pivot must appear in a different column, this means there are $d!$ different sequences of pivot positions.

To see why only the sequence of columns of the pivots matters, take some other $\Gamma'$ diagonal, and suppose that the sequence of pivot columns is the same for the two executions $(M, A\Gamma)$ and $(M, A\Gamma')$. We prove by induction on the $j$th pivot that all pivots are actually in the same positions. The first $j-1$ pivots appear in the same positions by assumption, so by Lemma \ref{persistent-scaling} the $j$th base-change matrix $N_{j}'$ for the input $(M, A\Gamma')$ equals $N_{j}\Gamma^{-1}\Gamma'$, where $N_{j}$ is the $j$th base-change matrix for the input $(M, A\Gamma)$. Then since the $j$th pivots appear in the same column, and scaling columns does not change the column entry of minimal valuation, the $j$th pivot will be in the same position for both executions as well.
\end{proof}

\begin{remark}
We note the similarity of Lemma \ref{covering-lemma} with \cite[Lemma 6.3]{S}, which states that any apartment $A$ in any building can be covered by the union of Weyl chambers based at some other fixed point $z$ with equivalence class in $\partial A$, the spherical apartment at infinity corresponding to $A$. We expect that Lemma \ref{covering-lemma} is an explicit analogue of this result in our specialized setting, where $\partial A$ is isomorphic to the symmetric group $S_d$ on $d$ elements, in which each Weyl chamber is replaced by a suitable apartment containing it to ensure the the convex hull of $z$ and $A$ is also covered.
\end{remark}

\section{Constructing enveloping membranes}

\label{convex-hulls}

In this section we combine the results of the previous section to solve the problem left open in Algorithm \ref{convex-hull-computation}. Namely, we present an algorithm to compute an enveloping membrane of a finite set of lattices. This allows us to realize convex hulls in the building as tropical polytopes.

\vspace{.1in}
\begin{algorithm}[List of apartments covering a convex hull]
\label{ApartmentList}
\begin{algorithmic}[1]
\item[]
\REQUIRE{$B_1,\dots, B_s$ base matrices for lattices $\Lambda_1,\dots, \Lambda_s$}
\ENSURE{A set of apartments covering $\conv(\Lambda_1,\dots, \Lambda_s)$}
\IF{$s = 2$}
 \RETURN{SA-basis of $\Lambda_1$ and $\Lambda_2$ via Algorithm \ref{CommonBasis}}
\ENDIF
\STATE $L_{s-1} \gets \text{set of apartments covering } \conv(\Lambda_2,\dots,\Lambda_s)$ via Algorithm \ref{ApartmentList}
\STATE $L_s\gets \emptyset$
\FORALL{$[A]\in L_{s-1}$}
  \STATE $L_A\gets \text{ set of apartments covering }\conv(\Lambda_1, [A])$ as in Remark \ref{convex-hull-with-apt-remark}
  \STATE $L_s\gets L_s \cup L_A$
\ENDFOR
\RETURN{$L_s$}
\end{algorithmic}
\end{algorithm}


\begin{theorem}
\label{apartment-list-correctness}
Let $M_1,\dots, M_s$ represent $s$ lattices $\Lambda_1,\dots, \Lambda_s$ in $\mathcal B_d$. Then Algorithm \ref{ApartmentList} correctly computes a list of apartments $L_{s}$ such that each lattice class $[\Lambda]\in\conv(\Lambda_1,\ldots,\Lambda_s)$ is contained in $[A]$ for some $[A]\in L_s$. Furthermore, $L_s$ has size at most $(d!)^{s-2}$.
\end{theorem}
Of course this theorem and Lemma \ref{membranes-apartments} together imply that Algorithm \ref{ApartmentList} can be used to compute an enveloping membrane for $\Lambda_1,\dots,\Lambda_s$. We simply concatenate all the matrices in the output $L_s$.
\begin{proof}
If the algorithm is correct, then $L_{s-1}$ contains at most $(d!)^{s-3}$ apartments by induction. Since $L_A$ has size at most $d!$ by Lemma \ref{covering-lemma}, $L_s$ has size at most $(d!)^{s-2}$. 

It remains to prove correctness. By Lemma \ref{min-conv-decomposition}, any lattice class $[\Lambda]$ in $\conv(\Lambda_1,\dots, \Lambda_s)$ is contained in $\conv(\Lambda_1, \Lambda')$ for some $[\Lambda']\in \conv(\Lambda_2,\dots, \Lambda_s)$. There exists some $[A]\in L_{s-1}$ such that $[\Lambda']\in [A]$, and so $[\Lambda]\in\conv(\Lambda_1, [A])$. In particular, there is some $[B]\in L_A$ such that $[\Lambda]\in [B]$.
\end{proof}

\begin{remark}
\label{convex-hull-with-apt-remark}
The crucial part of Algorithm \ref{ApartmentList} is computing the set $L_A$ of apartments covering $\conv(\Lambda_1,[A])$. Recall from Lemma \ref{covering-lemma} that this set is indexed by permutations in $S_d$. We sketch here how to compute the apartment corresponding to the identity permutation; all other apartments can be computed very similarly.

Let $\Gamma = \diag(\pi^{a_1},\dots,\pi^{a_d})$ be a diagonal matrix with indeterminates $a_1,\dots,\ a_d\in \mathbb Z$. First choose $\Gamma$ to be any diagonal matrix such that the first pivot for the first base change matrix $M_1^{-1}A\Gamma$ is in the first column. Next we compute the second base-change matrix; by decreasing both $a_1$ and $a_2$ by a large enough common value, Lemma \ref{persistent-scaling} guarantees that the first pivot will still be in the first column, and that the second pivot will appear in the second column. We next compute the third base-change matrix by reducing $a_1, a_2,$ and $a_3$ all by some large enough value, and so on.
\end{remark}

\begin{remark}
\label{better-bound}
We present in the next section a more efficient algorithm for the $s = 3$ case, Algorithm \ref{enveloping-membrane}, needing only $2^d$ apartments to cover the convex hull instead of $d!$. Because Algorithm \ref{ApartmentList} is inductive on $s$, Algorithm \ref{enveloping-membrane} can be used for the $s=3$ case, providing a slightly better overall bound of $2^d\cdot (d!)^{s-3}$ apartments needed to cover the convex hull of $s$ lattices.
\end{remark}

\begin{corollary}
\label{dimension-bound}
Let $\Lambda_1,\dots, \Lambda_s$ be lattices in $K^d$. Then their convex hull $\conv(\Lambda_1,\dots,\Lambda_s)$ is isomorphic to a tropical polytope in $\mathbb{TP}^{N}$ where $N \leq d\cdot 2^d \cdot (d!)^{s-3}-1$.
\end{corollary}
\begin{proof}
Remark \ref{better-bound} implies that a matrix $M$ with at most $d\cdot 2^d\cdot (d!)^{s-3}$ columns generates a membrane $[M]$ containing $\conv(\Lambda_1,\dots, \Lambda_s)$. Algorithm \ref{convex-hull-computation} then realizes the convex hull as a tropical polytope in a tropical projective space of dimension at most $d\cdot 2^d\cdot (d!)^{s-3}-1$.
\end{proof}

\begin{corollary}
\label{dual-polytope}
Let $\Lambda_1,\dots, \Lambda_s$ be lattices in $K^d$. Then their convex hull $\conv(\Lambda_1,\dots,\Lambda_s)$ is isomorphic to a tropical polytope spanned by at most $d\cdot 2^d\cdot (d!)^{s-3}$ points in $\mathbb{TP}^{s-1}$.
\end{corollary}
\begin{proof}
This follows directly from Corollary \ref{dimension-bound} and the self-duality of tropical polytopes.
\end{proof}

\begin{corollary}
\label{origin}
Let $\Lambda_1, \ldots, \Lambda_s$ be lattices in $K^d$. Let $[M]$ be the enveloping membrane for $\conv(\Lambda_1,\ldots,\Lambda_s)$ computed by concatenating the apartments from Algorithm \ref{ApartmentList}. Then $\Lambda_1$ is mapped to the origin by $\Psi_M$ in Algorithm \ref{convex-hull-computation}.
\end{corollary}
\begin{proof}
There is another representation of the building $\mathcal B_d$, which describes the vertices as additive norms $N:K^d\to\mathbb R\cup \{\infty\}$. We can easily pass between these two descriptions of the building in terms of lattice classes and additive norms. If $\Lambda$ is a lattice represented by a matrix $M$, then the corresponding additive norm is defined by
\[N_\Lambda(v)=\max\{u\in\mathbb Z: z^{-u}v\in\Lambda\}.\]
Write $M=(v_1,\ldots, v_n)$. By \cite[Lemma 21]{JSY}, the image of $\Lambda_1$ under the map of  Theorem \ref{minconv-isomorphism-theorem} is $(N_{\Lambda_1}(v_1),\ldots, N_{\Lambda_1}(v_n))$, where $N_{\Lambda_1}$ is the additive norm corresponding to $\Lambda_1$. But clearly $N_{\Lambda_1}(v_i)=0$ for each $i$, since each $v_i$ is an element for a basis for $\Lambda_1$.
\end{proof}

Viewed in the dual setting of Corollary \ref{dual-polytope}, Corollary \ref{origin} implies that our algorithm places us in the affine chart of $\mathbb{TP}^{s-1}$ where the first coordinate is zero.

\begin{example}
\label{conv-quadrilateral-ex}
Consider the following four $3\times 3$ matrices over $\mathbb C((t))$:
\[M_1 = \begin{pmatrix}1 & 1 & 1 \\ 1 & t & t^2\\ 1 & t^{-2} & t\end{pmatrix}, M_2 = \begin{pmatrix}1 & 1 & 1\\ t & t^2 & t^3\\ t^{-2} & t & t^5\end{pmatrix}, M_3 = \begin{pmatrix}1 & 1 & 1 \\ t^2 & t^3 & t^4\\ t & t^5 & t^8\end{pmatrix}, M_4 = \begin{pmatrix}1 & 1 & 1 \\ t^3 & t^4 & t^5 \\ t^5 & t^8 & t^{12}\end{pmatrix}.\]
These are the contiguous maximal submatrices of
\[M = \begin{pmatrix}1 & 1 & 1 & 1 & 1 & 1\\
1 & t & t^2 & t^3 & t^4 & t^5\\
1 & t^{-2} & t & t^5 & t^8 & t^{12}\end{pmatrix},\]
so the corresponding lattice classes certainly all lie in the membrane $[M]$.
An optimist could suppose that $[M]$ were in fact an enveloping membrane for the convex hull of our four matrices. Running through Algorithm \ref{convex-hull-computation} with the membrane $[M]$ yields the following tropical matrix:
\[\begin{pmatrix}
0 & 0 & 0 & 0 & 0 & 0\\
-2 & 0 & 0 & 0 & 0 & 0\\
-3 & -4 & 0 & 0 & 0 & 0\\
-6 & -8 & -5 & 0 & 0 & 0\end{pmatrix}.\]
The columns of this matrix span the tropical polytope $ P$, visualized using Polymake in Figure \ref{tconv1pic}. Its standard triangulation contains 18 vertices, 32 edges, and 15 triangles.

\begin{figure}[h!]
\centering
\begin{minipage}{0.475\textwidth}
\hspace*{-1cm}\includegraphics[width=1.2\textwidth]{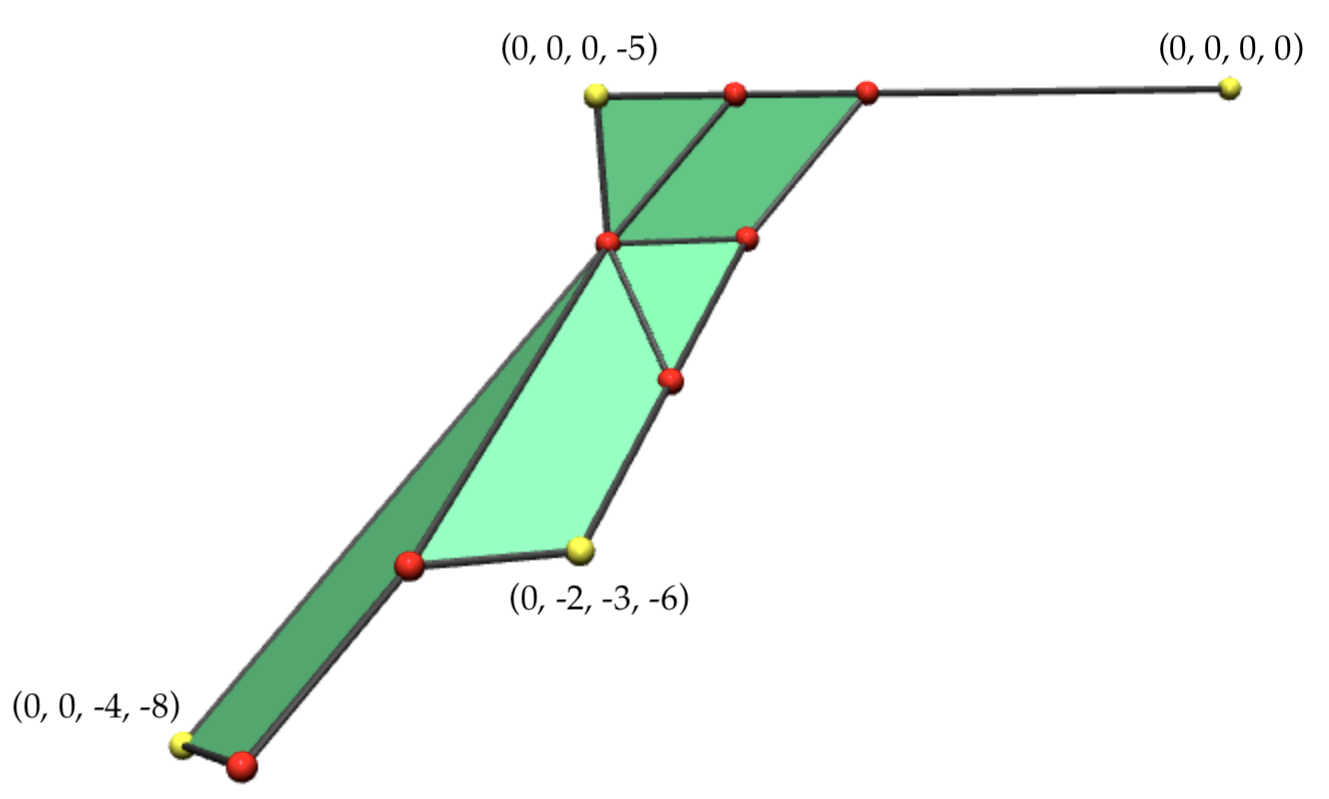}
\caption{The tropical polytope $ P$ obtained by using the membrane $[M]$ for the lattices $M_1, M_2, M_3$, and $M_4$ with Algorithm \ref{convex-hull-computation}. Points spanning the tropical convex hull are marked in yellow.}
\label{tconv1pic}
\end{minipage}
\hfill
\begin{minipage}{0.475\textwidth}
\hspace*{-.35cm}\includegraphics[width=1.2\textwidth]{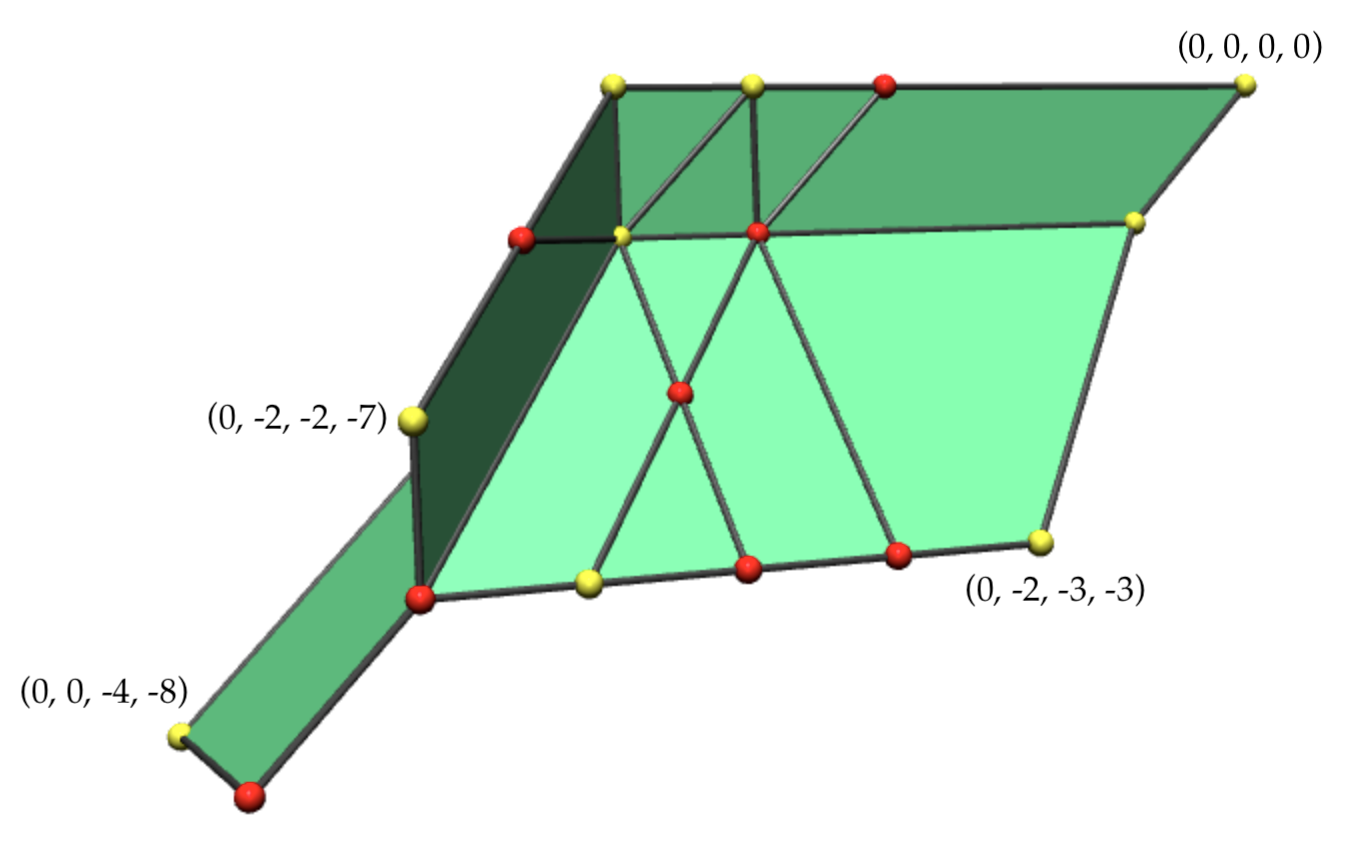}
\caption{The tropical polytope $ P'$ whose standard triangulation is isomorphic to the convex hull of $M_1, M_2, M_3, $ and $M_4$, with spanning vertices marked in yellow.}
\label{tconv2pic}
\end{minipage}
\end{figure}

However, when we run Algorithm \ref{ApartmentList} in Polymake to compute an enveloping membrane for $\conv(M_1, M_2, M_3, M_4)$, we obtain a different matrix $M'$ with 12 distinct columns.
Executing Algorithm \ref{convex-hull-computation} using the membrane $[M']$ yields that $\conv(M_1,M_2,M_3,M_4)$ is isomorphic as a simplicial complex to the tropical polytope $P'$ in Figure \ref{tconv2pic} spanned by
\[
\setcounter{MaxMatrixCols}{20}
\begin{pmatrix}
0 & 0 & 0 & 0 & 0 & 0 & 0 & 0 & 0 & 0 & 0 & 0  \\
-2 & 0 & 0 & 0 & -2 & 0 & -2 & 0 & 0 & -2 & 0 & 0\\
-3 & -4 & 0 & -1 & -2 & 0 & -2 & 0 & -1 & -3 & -1 & 0\\
-6 & -8 & -5 & -5 & -7 & 0 & -7 & -4 & -1 & -3 & -5 & 0\end{pmatrix}.\]

The standard triangulation of this polytope contains 29 lattice points, 67 edges, and 41 triangles. In particular, the convex hull of $M_1, M_2, M_3,$ and $M_4$ is larger than the polytope $ P$ obtained via the membrane $[M]$, even though each lattice spanning the convex hull is trivially contained in $[M]$. In turn, this means that $[M]$ does not contain the convex hull $\conv(M_1,M_2,M_3,M_4)$, demonstrating the fact that membranes are not convex.
\end{example}

\begin{example}
Let $K = \mathbb C((t))$ be the field of formal complex Laurent series, and let $M_1$ be the $4\times 4$ identity matrix, $M_2$ be diagonal with entries $1, t^3, t^{-2},$ and $t^{-2}$ respectively, and
\[
M_3 = \begin{pmatrix}
t^{-3}-t^2 & 1 - t^2 & - t^{-2} + 1 & t^{-2} - t\\
t^2 - t^3 & - t^{-3} + t & 1-t & 0\\
0 & -1+t & t^{-3} - t^3 & t^{-3} - 1\\
-t  + t^2 &- t^{-1}+1 & 0 & - t^{-1}+t^2
\end{pmatrix},\]
\[M_4 = \begin{pmatrix}
1 -t^3 & t^{-1} - 1 & 1 - t^2 & 1 - t^3\\
t^{-3} - 1 & 1 -t & 1 - t^2 & 1-t^3\\
-t^{-3} + t & -t^{-2} + 1 & -t^{-3} + t^{-1} & -1 + t\\
t^{-3}-t^{-2} & -t^{-1} + 1 & -t^{-1} + 1 & t^{-1}-1\end{pmatrix}.\]
\end{example}
Concatenating the matrices produced by Algorithm \ref{ApartmentList} applied to $M_1, M_2, M_3,$ and $M_4$ in Polymake gives a matrix $M$ with 84 distinct columns. Using the corresponding membrane $[M]$ with Algorithm \ref{convex-hull-computation}, we get a $4\times 84$ matrix over the tropical numbers. After pruning duplicate columns, we obtain the following matrix whose tropical row or column span gives the polytope displayed in Figure \ref{3dconvhull}. The triangulation of that polytope has 30 vertices, 95 edges, 102 triangles, and 36 tetrahedra.
\[
\setcounter{MaxMatrixCols}{20}
\begin{pmatrix}
0 &0 &0 &0 &0 &0 &0 &0 &0 &0 &0 &0 \\
-3 & -2 & -1 & -3 & -3 & 0 & 2 & 2 & 0 & -1 & -1 & -3\\
1 & 2 & 3 & 1 & 1 & 3 & 3 & 1 & 1 & 1 & 3 & 3\\
3 & 3 & 3 & 2 & 1 & 1 & 3 & 1 & 1 & 1 & 1 & 1
\end{pmatrix}.\]
\begin{figure}[h!]
\centering
\includegraphics[height=3in]{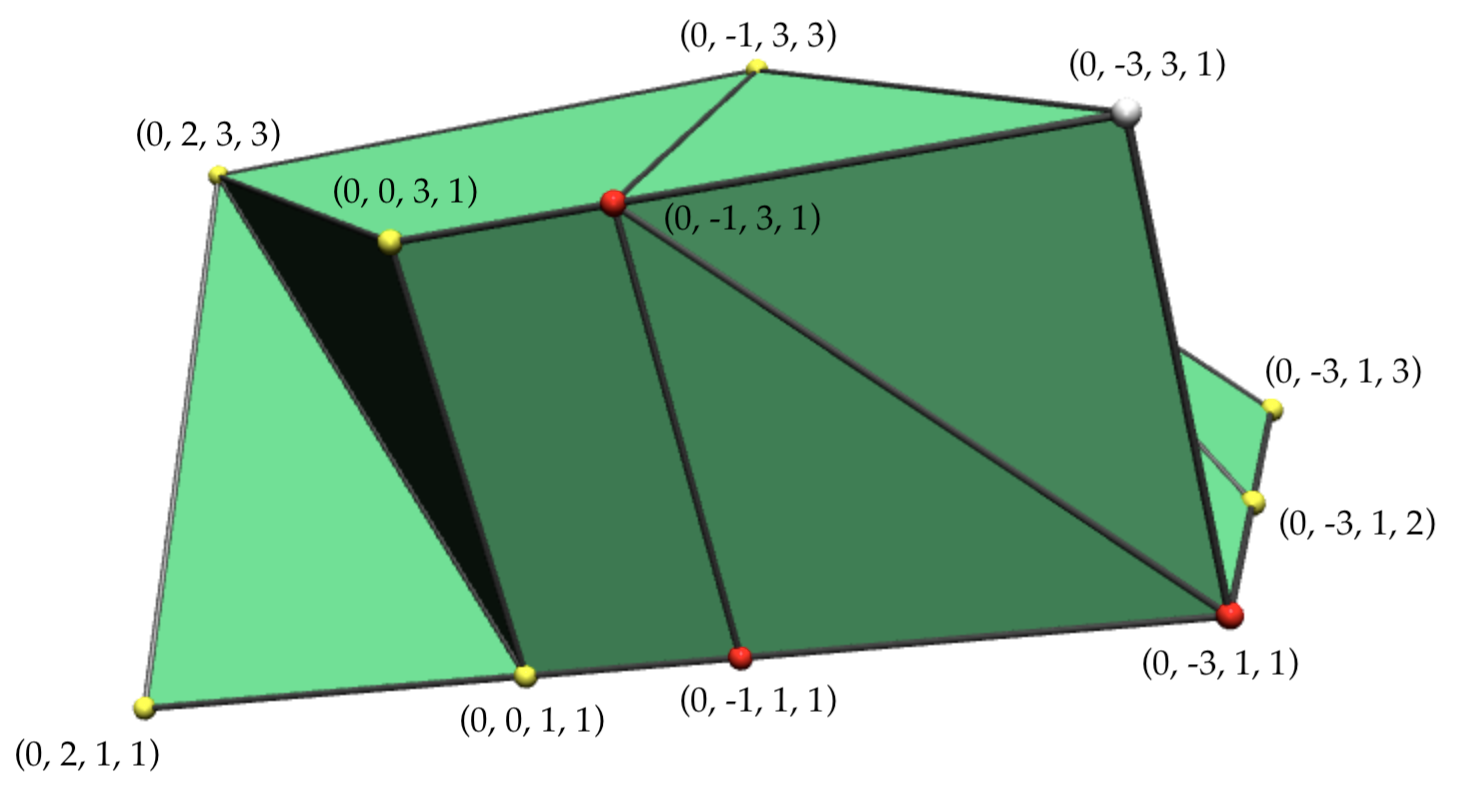}
\caption{The 3-dimensional tropical polytope isomorphic to the convex hull of our matrices $M_1,M_2,M_3,M_4$, whose standard triangulation has $f$-vector (30, 95, 102, 36).}
\label{3dconvhull}
\end{figure}
\section{Convex triangles}
\label{min-conv-triangles-section}

Suppose that $s = 3$, so that we wish to compute a \emph{convex triangle}: the convex hull of three lattice classes. This is relevant e.g. to \cite[Section 4.6]{CHSW}, which focuses on Mustafin varieties arising from convex triangles. In this case there exists a more efficient algorithm, taking advantage of the fact that $\conv(\Lambda_2,\Lambda_3)$ is just a path in the building. We now describe this improvement.

With some extra book-keeping, note that Algorithm \ref{CommonBasis} can output all of the following: 
\begin{itemize}
  \item an SA-basis $A$ which is a basis for $\Lambda_1$,
  \item a diagonal matrix $\Delta=\diag(\pi^{c_1},\ldots, \pi^{c_d})$ such that $A\Delta$ is a basis for $\Lambda_2$, where $c_1\leq \dots \leq c_d$,
  \item all of the base change matrices $N_1,\dots, N_d$,
  \item and the positions $p_1,\dots, p_d$ of the pivots $n_1,\dots, n_d$.
\end{itemize}
We justify the existence of such a $\Delta$. First, note that the base-change matrix $\Delta$ produced by Algorithm \ref{CommonBasis} can always be taken to be diagonal, since other monomial matrices correspond simply to reordering the scaled basis vectors of $\Lambda_1$. Second, we may reorder the columns of $A$ itself in any way we like; in particular, we can order them so that the matrix $\Delta$ has the structure described above.

\vspace{.1in}

\begin{algorithm}[Enveloping membrane for a convex triangle]
\,
\label{enveloping-membrane}
\begin{algorithmic}[1]
\item[]
\REQUIRE $M_1, M_2, M_3$ three $d\times d$ invertible matrices over $K$ whose columns are bases for lattices $\Lambda_1,\Lambda_2,\Lambda_3$ in $\mathcal B_d$
\ENSURE A list $L$ of apartments covering $\conv(\Lambda_1,\Lambda_2,\Lambda_3)$.
\STATE $L\gets \emptyset$
\STATE $(A, \Delta=\diag(\pi^{c_1},\ldots, \pi^{c_d})) \leftarrow$ $d\times d$ matrices such that $A$ is a basis for $\Lambda_2$ and $A\Delta$ is a basis for $\Lambda_3$, with $c_1\leq c_2\leq\ldots\leq c_d$
\FORALL{$i \in \{1,\dots, d-1\}$}
\STATE $\lambda \gets c_i$
\STATE $\Gamma_\lambda\gets \diag(\pi^{\max(\lambda, c_1)}, \dots, \pi^{\max(\lambda, c_d)})$
\STATE $t\gets 0$
\WHILE{$\lambda< c_{i+1}$}
	\STATE $\lambda\gets \lambda +t$
    \STATE $A_\lambda\gets$ an SA-basis for $(M_1, A\Gamma_\lambda)$
    \STATE $(N_1, \dots, N_d)\gets$ the sequence of base-change matrices in the SA-basis computation for $(M_1, A\Gamma_\lambda)$
    \STATE $(p_1,\dots, p_d)\gets$ the sequence of pivot positions in the SA-basis computation for $(M_1, A\Gamma_\lambda)$
    \STATE $L\gets L\cup \{[A_\lambda]\}$
	\STATE $t\gets c_{i+1}-c_i$
    \FORALL{$j\in \{1,\dots, d\}$ such that $p_j$ is in the first $i$ columns}
    	\STATE $v_1\gets$ valuation of $j$th pivot in $N_j$
        \STATE $v_2\gets$ least valuation among all elements of $N_j$ in columns $i+1, i+2, \dots, d$ not in positions $p_1, \dots, p_j$
        \STATE $t \gets \min(t, v_2-v_1+1)$
    \ENDFOR
\ENDWHILE
\ENDFOR
\RETURN{$L$}
\end{algorithmic}
\end{algorithm}

\begin{theorem}
\label{convex-triangles}
Let $M_1,M_2,$ and $M_3$ represent three lattices $\Lambda_1,\Lambda_2$, and $\Lambda_3$ in $\mathcal B_d$. Then Algorithm \ref{enveloping-membrane} correctly computes a list $L$ of apartments covering $\conv(\Lambda_1,\Lambda_2, \Lambda_3)$, where $L$ has size at most $2^{d}$.
\end{theorem}
As before, we can obtain an enveloping membrane for $\Lambda_1,\Lambda_2$, and $\Lambda_3$ by concatenating all matrices in $L$.
\begin{proof}
In the setup of Algorithm \ref{enveloping-membrane}, any class in $\conv(\Lambda_2,\Lambda_3)$ has a representative of the form $A\Gamma_\lambda$, where $\Gamma_\lambda = \diag(\pi^{\max(\lambda, c_1)}, \dots, \pi^{\max(\lambda, c_d)})$ and $\lambda$ is an integer between $c_1$ and $c_d$. It follows from Lemma \ref{min-conv-decomposition} that 
\[\conv(\Lambda_1,\Lambda_2,\Lambda_3) = \bigcup_{c_1\leq \lambda\leq c_d}\conv(M_1, A\Gamma_\lambda).\] 
We can therefore cover $\conv(\Lambda_1,\Lambda_2,\Lambda_3)$ with the apartments $[A_{\lambda}]$ containing $\conv(M_1, A\Gamma_\lambda)$ produced by Algorithm \ref{CommonBasis}. By Corollary \ref{pivot-positions}, furthermore, if we have computed $A_\lambda$ already we only need to compute $A_{\lambda+1}$ if some pivot changes position. 

Suppose this occurs, with $\lambda$ in the range $c_i\leq \lambda<c_{i+1}$. Then $B\Gamma_{\lambda+1}$ is obtained from $A\Gamma_{\lambda}$ by multiplying with the diagonal matrix whose first $i$ diagonal entries are $\pi$ and last $d-i$ diagonal entries are 1. Let $p_j$ be the earliest pivot which changes positions. By Lemma \ref{persistent-scaling}, it follows that the $j$th base-change matrix $N^{(\lambda+1)}_j$ for the pair $(M, A\Gamma^{\lambda+1})$ factors as $N^{(\lambda+1)}_j=N^{(\lambda)}_j\diag(\pi,\dots, \pi, 1, \dots,1)$, where $N^{(\lambda)}_j$ is the $j$th base-change matrix for the pair $(M, A\Gamma^{\lambda})$. Since the $j$th pivot differs for these two matrices, the $j$th pivot must appear in the first $i$ columns and there must be an element of equal valuation appearing in the last $d-i$ columns. Conversely, suppose there exists some $j$th pivot appearing in the first $i$ columns of $N_j^{(\lambda)}$ with an element of equal valuation in the last $d-i$ columns. Then either some earlier pivot already changed, or the $j$th pivot will be different for $N_j^{(\lambda+1)}$.

It follows that, for $\lambda$ in the range $c_i\leq \lambda < c_{i+1}$, we can quickly compute the smallest $t$ such that $(M_1, A\Gamma_{\lambda+t})$ will have some $j$th pivot in a different position than for $(M_1, A\Gamma_\lambda)$. For each $j$th pivot appearing in the first $i$ columns of $N_j^{(\lambda)}$, we can compare its valuation $v_1$ to the smallest valuation $v_2$ of all elements in the last $d-i$ columns of $N_j^{(\lambda)}$. If $p_j$ is the first pivot to change, it will change when $t= t_j:=v_2-v_1+1$. So $t=\min(t_j)$ is our desired increment. In particular, Algorithm \ref{enveloping-membrane} recomputes $A_\lambda$ each time a pivot changes, so it is indeed correct.

Next we prove that the list $L$ has size at most $2^d$. Suppose $\lambda$ is in the range $c_i\leq \lambda < c_{i+1}$. Our claim is that at most $\binom d i$ apartments are computed in this range, so that $\sum_i \binom d i = 2^d$ bounds the number of apartments in $L$. Write an $i$-sized subset $\sigma$ of $[d]$ as $(\sigma_1,\sigma_2,\dots, \sigma_i)$, where $\sigma_1<\sigma_2\dots<\sigma_i$. We can assign to each $\lambda$ an $i$-sized subset $\sigma^\lambda$ of $[d]$, where $j\in\sigma^\lambda$ if and only if the $j$th pivot appears in the first $i$ columns of $N_j$ when computing an SA-basis for $M_1$ and $A\Gamma_\lambda$. We can also define a well-ordering on the set of all $i$-sized subsets of $[d]$ lexicographically: $\sigma< \tau$ if and only if the first $j$ with $\sigma_j \neq \tau_j$ satisfies $\sigma_j<\tau_j$. The key insight is that $\sigma^{\lambda}< \sigma^{\lambda+1}$ if the corresponding pivot sequences for $\lambda$ and $\lambda+1$ differ. Since there are $\binom d i$ possible choices for $\sigma^\lambda$, there can be at most $\binom d i$ different pivot position changes for $\lambda$ in this range.

It remains to show why this key fact holds. Suppose that incrementing $\lambda$ by one changes some pivot position, with the $j$th pivot the first to change. The above analysis shows that the $j$th pivot for the pair $(M_1, A\Gamma_\lambda)$ must be in the first $i$ columns, and that this must change for the pair $(M_1, A\Gamma_{\lambda+1})$. It follows that $j$ must be in $\sigma^{\lambda}$, and that $j$ cannot be in $\sigma^{\lambda+1}$. Furthermore, because $j$ is the first pivot to change, for each $\ell<j$ we have $\ell\in \sigma^\lambda \iff \ell \in \sigma^{\lambda+1}$. Hence $\sigma^\lambda <\sigma^{\lambda+1}$, as desired.
\end{proof}

\begin{example}
Fix $K = \mathbb Q_3$ and the building $\mathcal B_5$. Let $M_1$ be the $5\times 5$ identity matrix, and let each entry of $M_2$ and $M_3$ be sampled uniformly at random from the finite set $\{3^e: e\in \mathbb Z, -20\leq e \leq 20\}$. The author took 1000 such triangles and computed enveloping membranes via Algorithm \ref{enveloping-membrane} in Mathematica. After pruning duplicate columns, the matrices describing the enveloping membranes always had at least 6 columns, and at most 25. For comparison, the upper bound implied by our Algorithm \ref{enveloping-membrane} is $5\cdot 32 = 160$ columns. A histogram describing the frequency counts for the size of the membranes is presented in Figure \ref{hist}. 

\begin{figure}[h]
\includegraphics[height=2.55in]{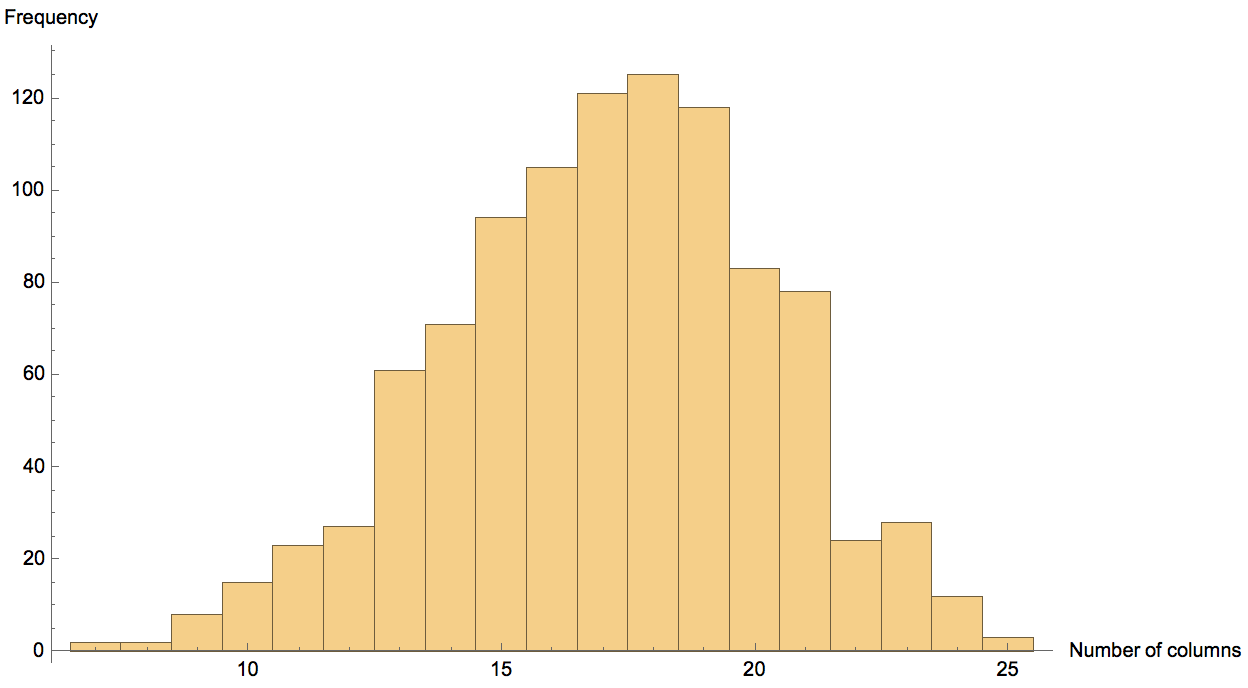}
\caption{Frequency counts for the number of columns of enveloping membranes produced by Algorithm \ref{enveloping-membrane} for random convex triangles.}
\label{hist}
\end{figure}
One example of a convex triangle attaining the maximal number of columns is given by
\[M_2=\begin{pmatrix}
3^{-15} & 3^{16} & 3^{-7} & 3^{-8} & 3^{-13}\\
3^{-13} & 3^{20} & 3^{-12} & 3^{-9} & 3\\
3^{-19} & 3^{19} & 3^7 & 3^{-15} & 3^{10}\\
3^9 & 3^{-12} & 3^{-12} & 3^{-17} & 3^{-18}\\
3^{-17} & 3^{-4} & 3^{-7} & 3^{-3} & 3^{20}
\end{pmatrix},
M_3 = \begin{pmatrix}
3^{-1} & 3^{-8} & 3^{-20} & 3^{-1} & 3^{-20}\\
3^{10} & 3^6 & 3^0 & 3^2 & 3^{-20}\\
3^{-6} & 3^8 & 3^3 & 3^5 & 3^{-13}\\
3^{-15} & 3^9 & 3^{-9} & 3^2 & 3^{-7}\\
3^{12} & 3^{-3} & 3^5 & 3^{-16} & 3^{-13}
\end{pmatrix}.\]
After applying Algorithm \ref{enveloping-membrane} to obtain an appropriate membrane $[M]$, the author computed the tropical polytope via Algorithm \ref{convex-hull-computation} presented in Figure \ref{tconvpic3}.
\begin{figure}[h]
\hspace*{-1.35cm}\includegraphics[height=3.1in]{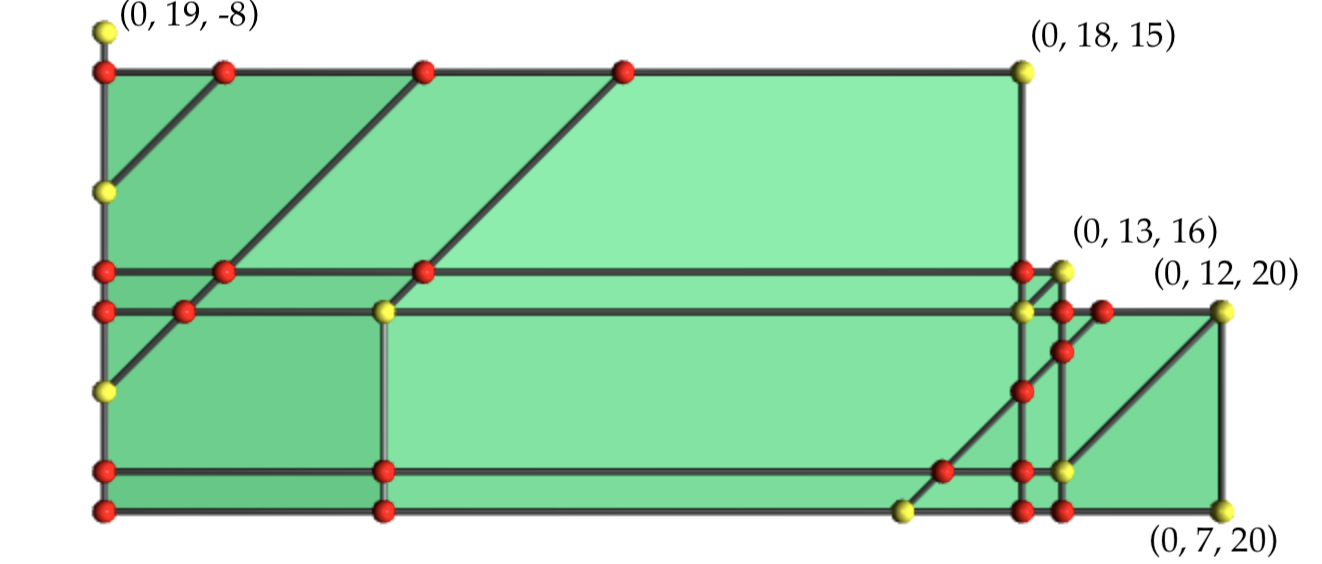}
\caption{The tropical polytope isomorphic to the convex hull of $M_1, M_2, M_3$, with spanning vertices in yellow. Note that the $x$- and $y$-axes have been flipped.}
\label{tconvpic3}
\end{figure}
Note that this convex hull can be spanned by only five of the given points: $(0, 19, -8)$, $(0, 18, 15)$, $(0, 13, 16), (0, 12, 20),$ and $(0, 7, 20)$. Let $M'$ be the square submatrix of $M$ with columns corresponding to these five points. Running through Algorithm \ref{convex-hull-computation} using the apartment $[M']$ yields a coarser subdivision of the same tropical polytope. This implies that the convex hull of our three matrices $M_1, M_2,$ and $M_3$ all lie in the single common apartment $[M']$, which can also be seen using \cite[Lemma 25]{JSY}. That our algorithms do not notice this fact suggests that they likely can be improved.
\end{example}

\bibliographystyle{amsalpha}
\bibliography{master.bbl}

\providecommand{\bysame}{\leavevmode\hbox to3em{\hrulefill}\thinspace}
\providecommand{\MR}{\relax\ifhmode\unskip\space\fi MR }
\providecommand{\MRhref}[2]{%
  \href{http://www.ams.org/mathscinet-getitem?mr=#1}{#2}
}
\providecommand{\href}[2]{#2}
\begin{thebibliography}{CHSW11}

\bibitem[AB08]{AB}
Peter Abramenko and Kenneth~S. Brown, \emph{Buildings: Theory and
  applications}, Graduate Texts in Mathematics, vol. 248, Springer, New York,
  2008.

\bibitem[CHSW11]{CHSW}
Dustin Cartwright, Mathias H\"{a}bich, Bernd Sturmfels, and Annette Werner,
  \emph{Mustafin varieties}, Selecta Math. (N.S.) \textbf{17} (2011), no.~4,
  757--793.

\bibitem[DS04]{DS}
Mike Develin and Bernd Sturmfels, \emph{Tropical convexity}, Doc. Math.
  \textbf{9} (2004), 1--27.

\bibitem[DT98]{DT}
Andreas Dress and Werner Terhalle, \emph{The tree of life and other affine
  buildings}, Proceedings of the {I}nternational {C}ongress of
  {M}athematicians, {V}ol. {III} ({B}erlin, 1998), vol. III, 1998,
  pp.~565--574.

\bibitem[Fal01]{F}
Gerd Faltings, \emph{Toroidal resolutions for some matrix singularities},
  Moduli of abelian varieties ({T}exel {I}sland, 1999), Progr. Math., vol. 195,
  Birkh\"{a}user, Basel, 2001, pp.~157--184.

\bibitem[GJ00]{GJ}
Ewgenij Gawrilow and Michael Joswig, \emph{polymake: a framework for analyzing
  convex polytopes}, Polytopes---combinatorics and computation ({O}berwolfach,
  1997), DMV Sem., vol.~29, Birkh\"{a}user, Basel, 2000, pp.~43--73.

\bibitem[Hir18]{H}
Hiroshi Hirai, \emph{Computing degree of determinant via discrete convex
  optimization on euclidean building}, 2018, arXiv:1805.11245.

\bibitem[Hit11]{S}
Petra Hitzelberger, \emph{Non-discrete affine buildings and convexity}, Adv.
  Math. \textbf{227} (2011), no.~1, 210--244.

\bibitem[HL17]{HL}
Marvin~Anas Hahn and Binglin Li, \emph{Mustafin varieties, moduli spaces and
  tropical geometry}, 2017, arXiv:1707.01216.

\bibitem[Jos]{J}
Michael Joswig, \emph{Essentials of tropical combinatorics}, in preparation,
  \url{http://page.math.tu-berlin.de/~joswig/etc/index.html}, accessed in 2018.

\bibitem[JSY07]{JSY}
Michael Joswig, Bernd Sturmfels, and Josephine Yu, \emph{Affine buildings and
  tropical convexity}, Albanian J. Math. \textbf{1} (2007), no.~4, 187--211.

\bibitem[KT06]{KT}
Sean Keel and Jenia Tevelev, \emph{Geometry of {C}how quotients of
  {G}rassmannians}, Duke Math. J. \textbf{134} (2006), no.~2, 259--311.

\bibitem[MS15]{MS}
Diane Maclagan and Bernd Sturmfels, \emph{Introduction to tropical geometry},
  Graduate Studies in Mathematics, vol. 161, American Mathematical Society,
  Providence, RI, 2015.

\bibitem[PRS13]{PRS}
Georgios Pappas, Michael Rapoport, and Brian Smithling, \emph{Local models of
  {S}himura varieties, {I}. {G}eometry and combinatorics}, Handbook of moduli.
  {V}ol. {III}, Adv. Lect. Math. (ALM), vol.~26, Int. Press, Somerville, MA,
  2013, pp.~135--217.

\bibitem[Rin13]{R}
Felipe Rinc\'{o}n, \emph{Local tropical linear spaces}, Discrete Comput. Geom.
  \textbf{50} (2013), no.~3, 700--713.

\bibitem[Sup08]{Su}
Tharatorn Supasiti, \emph{Serre's tree for {$SL_2(\mathbb F)$}}, unpublished
  honors thesis, 2008.

\end{thebibliography}
\end{document}